\documentclass[12pt]{article}

\usepackage[utf8]{inputenc}
\usepackage{amsmath, amssymb, amsthm}
\usepackage{xcolor}
\usepackage{geometry}
\usepackage{hyperref,cite}
\usepackage{subfig} 
\usepackage{multicol}
\usepackage{blindtext}
\usepackage{mathrsfs}
\usepackage{stmaryrd}
\usepackage{tikz-cd}
\usepackage{booktabs}
\usepackage{cite}
\usepackage{float}
\newtheorem{theorem}{Theorem}[section]
\newtheorem{lemma}[theorem]{Lemma}
\newtheorem{corollary}[theorem]{Corollary}
\newtheorem{proposition}[theorem]{Proposition}
\theoremstyle{definition}
\newtheorem{definition}[theorem]{Definition}

\theoremstyle{remark}
\newtheorem{remark}[theorem]{Remark}

\newcommand{\R}{\mathbb{R}}

\newcommand{\Z}{\mathbb{Z}}

\newcommand{\rT}{{\sf S}}
\newcommand{\cT}{{\sf T}}

\newcommand{\rG}{{\sf H}}
\newcommand{\dis}{(s_0^{10}-s_4^{10})} 
\newcommand{\hS}{\hat{S}}
\newcommand{\hD}{\hat{\Delta}}
\newcommand{\hK}{\hat{K}}

\title{Gauss-Manin Connection in Disguise: Holomorphic Anomaly Equations and Higher Genus Real Gromov-Witten Invariants}
\author{Felipe Espreafico }
\date{\today}
\begin{document}
\maketitle
\begin{abstract}
    We construct an algebraic framework for the extended holomorphic anomaly equations following ideas of Alim, Movasati, Scheidegger and Yau. We choose the quintic threefold as our main model and use this framework to solve the equations and fix the ambiguity up to genus 5. We obtain predictions for real Gromov-Witten invariants, or, equivalently, real Gopakumar-Vafa invariants that are integers and agree with their complex counterparts modulo 2. In the end, we discuss how the framework could be extended and compute the same predictions for the bicubic up to genus 9. 
\end{abstract}
\section{Introduction}

Mirror symmetry is a very celebrated idea in Physics and Mathematics. It started as a duality between Calabi Yau threefolds, relating the variation of Kähler structure of one threefold to the variation of complex structure of another threefold, its {\it mirror manifold}. The name comes from the fact that mirror manifolds have mirror-symmetric Hodge diamonds.

The first applications of such theory in Mathematics were in enumerative geometry, when virtual counts of rational curves on Calabi-Yau threefolds were shown to be related to period integrals on the mirror threefold. These counts of curves are computed via integrals on the moduli space of stable maps and are called Gromov-Witten invariants. The manifold whose invariants we are interested in is called $A$-model, while the manifold where the periods are computed is referred to as the $B$-model. The first of such predictions, for rational curves, was done in~\cite{candelas_pair_1991} for the quintic threefold. The theory was extended to treat counts of higher genus curves via a recursive method, the BCOV holomorphic anomaly equations, which made it possible to deduce predicted generating functions for higher genus counts (called \textit{topological string amplitudes}) in terms of lower genus ones, up to a holomorphic ambiguity~\cite{bershadsky_holomorphic_1993,bershadsky_kodaira-spencer_1994}. Of course, in order to get acurate predictions for invariants, one needs to find a way to fix such ambiguity. 

The work of Yamaguchi and Yau~\cite{yamaguchi_topological_2004}, showed that the solutions of these equations have certain polynomiality properties that imply that the ambiguities can be determined by a finite number of conditions. This led to the work~\cite{huang_topological_2009}, where a way to fix this ambiguity was worked out for some examples. They use particular conditions on how the amplitudes behave in each of the special points of the B-model variation of complex structure and classical algebro-geometric results on the nonexistence of curves of certain genus and degrees. In particular, for the quintic threefold, such numbers were predicted up to genus 51. It is important to remark that the BCOV anomaly equation, the polynomiality and the other conditions used to compute the topological string amplitudes are intrinsic to the B-side geometry (except for the genus-degree results mentioned), and there is no proof that, in general, the amplitudes are actualy computing Gromov-Witten invariants. For the quintic, this was recently proved in the works~\cite{guo_structure_2018,chang_polynomial_2021,chang_conifold_2026, chang_bcov_2026}.

An interesting interpretation for the geometry of the $B$-model came in the work of Alim, Movasati, Scheidegger and Yau~\cite{alim_gaussmanin_2016}, based on earlier works by Movasati for the elliptic curve and the quintic threefold~\cite{movasati_ramanujan_2012, movasati_modular-type_2015}. The authors show that, for Calabi-Yau threefolds, one can consider the complex structure moduli space enhanced with the de Rham cohomology bundle and that this space gives rise to special vector fields and the holomorphic anomaly equation can be expressed in terms of such vector fields. This special vector fields form a Lie Algebra, called AMSY Lie Algebra, that encodes all the structure from the anomaly equation. The interesting aspect of such construction is that, besides being purely algebraic and allowing future arithmetic generalizations, it also provides a clear analogy between the polynomiality enjoyed by the generating functions and the classical theory of modular forms. Some other works explored this ideas for other cases than Calabi-Yau threefold, as~\cite{movasati_gaussmanin_2021,nikdelan_modular_2020},\cite{alim_algebra_2022} and can attest for the generality of these contructions. 

Parallel to the development of counts of complex curves in Calabi Yau threefolds, the real case was also considered. For the real quintic, the genus zero counts were first computed by Walcher~\cite{walcher_opening_2007}. These are virtual counts of rational curves with an involution (or complex curves with boundary on the real Lagrangian). In that case, the periods considered on the mirror now involve a relative period, i.e., an integral with boundary on a pair of holomorphic curves. Walcher also later developed an extended version of the BCOV holomorphic anomaly equation~\cite{walcher_extended_2009, walcher_evidence_2009}, for which polynomiality was proved in~\cite{alim_polynomial_2007}. Notice that the first generating functions considered by Walcher (in~\cite{walcher_extended_2009}) included only ``oriented'' curves with boundary, which excluded curves that contribute to the real count. This was corrected in the second paper~\cite{walcher_evidence_2009}, but the equation satisfied ended up being the same up to initial conditions. Such virtual counts of real curves were later made rigorous in the work of Georgieva and Zinger~\cite{georgieva_real_2015, georgieva_real_2018} as real Gromov-Witten invariants in all genera. However, in contrast to the complex case, no computation was available for compact models in genus higher than two. For toric models, as local $\mathbb P^2$, the ambiguity was fixed via the real topological vertex, mainly developed in~\cite{krefl_real_2009, krefl_real_2010, krefl_extended_2011}. The work~\cite{marino_non-perturbative_2026} also studies the holomorphic anomaly equation for local $\mathbb P^2$ in the real case, using resurgence techniques.

In this paper, our goal is twofold. First, we follow the ideas of~\cite{alim_gaussmanin_2016} to give an algebraic construction of the higher genus B-model of the real mirror quintic, involving the algebraic de Rham cohomology and Hodge theoretical considerations. The main difference compared to the complex situation is that we have to consider relative de Rham cohomology and mixed Hodge structures. The first part of this program, involving genus zero, was carried out by the author in~\cite{espreafico_gauss-manin_2026}. There, we constructed the moduli space of enhanced mirror quintics and proved that one can recover the count of disks (real genus zero curves) from the Gauss-Manin connection. In this paper, we finish the program for the real mirror quintic by defining a Lie algebra of vector fields and showing that the extended holomorphic anomaly equation can be expressed in terms of these vector fields.  

The second goal is to use the framework above to solve the extended anomaly equations and fix the ambiguity for the first few genera. We use the gap condition at the conifold point and the vanishing of complex invariants to be able to compute Gopakumar-Vafa invariants up to genus 5 for the quintic threefold. Assuming that $n_{6,5}^{\mathbb R}$ is computable, we can go up genus 8. In the end, we explain how to extend this program to other models, which can be done once we have an holomorphic curve that correspond to the extra period, the role played by a pair of conics for the quintic. The computations, however, can be done starting only from the Picard-Fuchs data and we carry on the program for the bicubic, obtaining the real GV invariants up to genus 9. Assuming we know $n^{\mathbb R}_{10,9}$, we can go up to genus 12. Our numbers are all integers and agree with the complex ones modulo 2. The proposed computation of $n_{6,5}^{\mathbb R}$ for the quintic and $n^{\mathbb R}_{10,9}$ for the bicubic use the comparision between GV and DT invariants. An arithmetic version of similar predictions was considered in~\cite{espreafico_motivic_2024}.

Although the algebraic framework is not particularly important for the computation, i.e., it would be possible to do it only from the equations in~\cite{alim_polynomial_2007}, it gives a nice point of view and simplifies the expansions in the holomorphic limits, since those are solutions of a vector field. Some other works in the literature gave solutions integrating the anomaly equation, but, up to our knowledge, this is the first work where the ambiguities were fixed for the compact models beyond genus 2.

\subsection{Main Results}
The main theorem in the paper is the following, which is an algebraic version of the extended holomorphic anomaly equation, in the style of~\cite{alim_gaussmanin_2016}, for the real quintic.

\begin{theorem}\label{alg-HAE-GMCD-Int}
    The B-side topological string partition functions $G^{\chi}$, $\chi>0$  of Euler characteristic $\chi$ can be seen as regular functions on $\rT$ in the variables $s_0,\dots,s_8$ given as solutions to the recursive differential equations:
    \begin{align}
        {\sf R}_i G^{\chi} = 0, i=1,3\\
        {\sf R}_2 G^{\chi} = \chi G^{\chi}\\
        {\sf R}_4 G^{\chi} = \frac1{5^8} \left(\sum_{\substack{\chi_1+\chi_2 = \chi-2\\ \chi_1,\chi_2\geq0}}{\sf R}_0G^{\chi_1}{\sf R}_0G^{\chi_2}+{\sf R}_0^2G^{\chi-2}\right)\\
        {\sf R}_7 G^{\chi} = -\frac1{5^8}{\sf R}_0 G^{\chi-1}
    \end{align}
    \end{theorem}
\bigskip

Here, $\rT$ is the space parameterizing mirror quintics with a pair of conics $C_{\pm}$ and a basis of the relative de Rham cohomology $H_{\rm dR}^3(Y,C_{+}\cup C_{-})$ that respects the Hodge structure and have constant intersection product (see definition~\ref{relat-enha-mq} and~\cite{espreafico_gauss-manin_2026}). The $\sf R_i$ are vector fields on $\rT$, defined purely algebraically and form a Lie algebra over $\mathcal O_{\rT}$ (see~\ref{thm:lie-algebra}).

Theorem~\ref{alg-HAE-GMCD-Int} only allows computation of $G^{\chi}$ up to an ambiguity that depends on $s_4$ and $s_0$. 

\begin{theorem}\label{thm:ambiguity-intro}
    Assuming standard growth conditions (see~\ref{amb-ansatz}), the ambiguity to the solutions of the extended holomorphic anomaly equation has the form $s_0^{2\chi}h^\chi$, where
    \begin{equation}\label{eq:amb-ansatz}
h^{\chi}(s_0,s_4) \;=\; \frac{1}{s_0^{5\chi}}\!\!\sum_{\substack{0\le k\le 3\chi\\ k\equiv\chi\ (2)}}\!\!
  b_k\, s_4^{5k}\,\bigl(s_0^{10}-s_4^{10}\bigr)^{\frac{\chi-k}{2}} .
\end{equation}

The unknown coefficients $b_k$ can be fixed up to $\chi = 4$ by assuming the gap condition around the conifold, the expression for the contribution of constant maps for even $\chi$ and the fact that the vanishing of complex invariants implies vanishing of the real invariants. 
\end{theorem}

\subsection{Outline}

The paper is organized in four different parts. Section~\ref{sec:Basics} recalls the basic framework, introducing the Picard-Fuchs operator, the mirror map, the holomorphic limit expansions around the special points and the GMCD (Gauss-Manin connection in Disguise) framework from~\cite{espreafico_gauss-manin_2026}. We also explain how all the expansions can be obtained from the modular vector field $\sf R_0$.

Section~\ref{sec: Alim-Lange} recalls the basic special geometry framework that is essential to defining the higher genus amplitudes and the anomaly equation. We recall the extended anomaly equation as in~\cite{walcher_evidence_2009}. After that, we also explain how to write the anomaly equation in terms of the propagators and terminators as done in~\cite{alim_polynomial_2007}.

Section~\ref{sec: algebraic Hae} is where we formulate the algebraic framework for the extended anomaly equation. We explain the relationship between the generators from special geometry (from~\cite{alim_polynomial_2007}) and the coordinates on ${\rT}$. After that we define the vector fields $\sf R_i$, show that they form a Lie Algebra and prove Theorem~\ref{alg-HAE-GMCD-Int}. 

Section~\ref{sec:ambiguity} explains how to fix the ambiguity and the algorithm to compute the real Gopakumar-Vafa invariants, which proves Theorem~\ref{thm:ambiguity-intro}. We also include explicit calculations for $\chi=1$.

In the last Section~\ref{sec:conclusions}, we quickly discuss the extension for the bicubic and other models. We present the algebraic version of the extended anomaly equation for the bicubic and use it to fix the ambiguity in that case. We finish
by discussing further research directions.

There is also an Appendix in the end, where the reader can find tables of real GV invariants up to degree 30 for the quintic and bicubic.
\bigskip

\textbf{Acknowledgments.} Most of this work was developed during my postdoctoral position at IMJ-PRG, Sorbonne Université, in Paris, funded under the ERC project ROGW 864919. The final bit of writing was done after I moved to TU Darmstadt, funded by a DFG Walter Benjamin position ES 835/1-1.

I want to thank Penka Georgieva for the many discussions regarding this topic and all the support during my time in Paris. I also need to thank Johannes Walcher and Hossein Movasati for the hospitality during my visits and helpful discussions.

\textbf{LLM use and Data availability.} The computations were done in Mathematica and the codes as well as further numbers are available in the corresponding section of the author's website {\tt espreafico-guelerman.github.io}. A first code was written by the author to compute genus 2 invariants for the quintic. LLMs were used to correct small errors in the original code and to generalize it for higher genera. In particular, the code for computing the invariants for the bicubic was generated by LLMs from the original code for the quintic. The latex codes for the tables of numbers and long numerical expressions were also LLM-generated. The author takes full responsability for the text and mistakes on the paper.

\section{The Mirror Quintic, Picard-Fuchs equation and genus-zero expansions}\label{sec:Basics}

\subsection{The Mirror Quintic and Picard-Fuchs equation at \texorpdfstring{$z=0$}{z=0}}\label{pic-fuchs}
We start by recalling the basic setup of Mirror Symmetry for the real quintic threefold. Let $X$ be a quintic threefold in $\mathbb P^4$ defined over the real numbers. There is a natural involution on $X$ given by complex conjugation, whose fixed locus is given by the set of real points of $X$, denoted $X_{\mathbb R}$. $X_{\mathbb R}\subset X$ is a Lagrangian submanifold. One can define the moduli space of stable maps $M_{g,0}(X, X_{\mathbb R}, d)$ of stable maps of genus $g$, degree $d$ and with boundary on $X_{\mathbb R}$. Equivalently, one can also consider the space of stable maps $M_{g,0}^{\tau}(X,d,\sigma)$ of stable maps of genus $g$ that preserve antiholomorphic involutions $\tau$ on $C$ and $\sigma$ on $X$. The real Lagrangian is identified with the fixed locus of $\sigma$ and the boundary of the curve with the fixed locus of $\tau$. In this context, one can compute real Gromov-Witten invariants of any genus and degree by integrating over the virtual fundamental class of the moduli spaces. 

The mirror of a quintic threefold is the mirror quintic family, given by the resolution of singularities of
\begin{equation}\label{eq:mirror-quintic}
    Y_\psi:\quad \{x_0^5+x_1^5+x_2^5+x_3^5+x_4^5-5\psi\,x_0x_1x_2x_3x_4=0\}/\mathbb{Z}_5^3
\end{equation}
where $\mathbb Z_5$ acts naturally by roots of unity. We denote the elements of such family as $Y_{\psi}$. The B-model is mainly interested by the periods computed from the third (algebraic) de Rham cohomology of $H^3(Y_{\psi})$. 
Those periods satisfy the Picard-Fuchs differential equation
\begin{equation}\label{eq:PF}
  \mathcal D=\theta^4-5^5 z\left(\theta+\tfrac15\right)\left(\theta+\tfrac25\right)
  \left(\theta+\tfrac35\right)\left(\theta+\tfrac45\right).
\end{equation}
where $z:=(5\psi)^{-5},\, \theta:=z\frac{d}{dz}=-\tfrac15\,\psi\frac{d}{d\psi}$. 
Note that, in some references, the letter $z$ is used for ${\left(\psi\right)}^{-5}$ instead. The equation can be solved around $z=0$ by the Frobenius method giving us hypergeometric solutions as follows:
\begin{equation}\label{eq:phij}
  \varphi_j(z)=\frac{1}{(2\pi i)^j}\,\frac{1}{j!}\,
  \frac{\partial^{\,j}F(z,s)}{\partial s^{\,j}}\bigg|_{s=0},
  \qquad j=0,1,2,3 .
\end{equation}
where 
\begin{equation}\label{eq:frobenius}
  F(z,s)=\sum_{n\ge 0}\frac{\Gamma\big(5(n+s)+1\big)}{\Gamma(n+s+1)^5}\,
  \frac{\Gamma(1+s)^5}{\Gamma(1+5s)}\;z^{\,n+s},
\end{equation}

The holomorphic solution is $\varphi_0$ and is given by:
\begin{equation}\label{eq:varpi0}
  \varpi_0(z)=\varphi_0(z)=\sum_{n\ge0}\frac{(5n)!}{(n!)^5}z^n
  =1+120\,z+113400\,z^2+168168000\,z^3+305540235000\,z^4+\dots
\end{equation}
The mirror map, given by \(q = e^{2\pi i t}\), where \(t=\frac{\varphi_1}{\varphi_0}\) relates the variable \(z\) and the variable \(q\), which encodes the quantum cohomology of the A-model. We have
\begin{subequations}\label{eq:qz-both}
\begin{equation}\label{eq:qz}
  q(z)=z+770\,z^2+1014275\,z^3+1703916750\,z^4+3286569025625\,z^5+O(z^6),
\end{equation}
\begin{equation}\label{eq:zq}
  z(q)=q-770\,q^2+171525\,q^3-81623000\,q^4-35423171250\,q^5+O(q^6).
\end{equation}
\end{subequations}

The classical prediction from mirror symmetry, dating back to~\cite{candelas_pair_1991} is that the genus zero Gromov-Witten invariants of the quintic threefold can be computed from such periods in the following way: 
\begin{equation}\label{eq:genus0}
  \frac52\left(\frac{\varphi_1\varphi_2}{\varphi_0^{\,2}}-\frac{\varphi_3}{\varphi_0}\right)
  =\frac{1}{(2\pi i)^3}\left[\frac{5}{6}\big(\log q\big)^3+\sum_{d\ge1}N_{0,d}\,q^d\right].
\end{equation}

Such predictions were later confirmed by Givental by using torus localization. We usually call these expansions around $z=0$ \textit{large complex structure limit}, since the monodromy there is maximally unipotent. 

For real Gromov-Witten invariants, the B-model picture consists of the mirror quintic family and a family of pairs of holomorphic curves:
\begin{equation}\label{eq:deligne-conics-intro}
C_{\pm}:=\{x_0+x_1=x_2+x_3=x_4^2\pm\sqrt{5\psi}x_1x_3=0\}\subset Y_{\psi}
\end{equation}

The pair of curves induces a cycle in the relative homology \(H_{3}(Y_{\psi},C_{\pm})\) connecting the two curves. That cycle gives origin to a new period $\tau$, that satisfies an extended version of the Picard-Fuchs equation
\begin{equation}\label{eq:ext-PF}
(2\theta-1)\mathcal D\tau=0
\end{equation}
which translates to an inhomogeneous equation $\mathcal D\tau = c\sqrt z$, for a constant $c$.

The extra solution of the extended equation allows us to compute the generating function of real Gromov-Witten invariants in genus zero~\cite{pandharipande_disk_2008}. If we take the solution $\tau(z)$ satisfying  
\begin{equation}\label{eq:ext-PF-source}
\mathcal D(\tau(z))=\frac{15\sqrt{z}}{8}
\end{equation}
and expand $\frac{\tau}{\varphi_0}$ in the coordinate $q$, we get
\begin{equation}\label{eq:disk-GW}
\frac{\tau}{\varphi_0} =2\sum_{d\, {\rm odd}}N^{\mathbb R}_{0,d} q^{d/2}= 2\left(15 \sqrt{q} + \frac{2300 q^{3/2}}{3} + \frac{2720628 q^{5/2}}{5} + \frac{23911921125 q^{7/2}}{49} + \dots\right)
\end{equation}

Notice that real Gromov-Witten invariants are always half of the disk count that was computed in~\cite{walcher_opening_2007} and~\cite{pandharipande_disk_2008}. This is simply because each real curve can be given as the boundary of two different disks.

\subsection{The conifold expansion}\label{sec: conif}

One can also think about the solution of the Picard-Fuchs around the conifold point $z=5^{-5}$. We start with a basis of series solutions in terms of $x=\psi^5-1=\frac{1}{5^5z}-1$ for the homogeneous equation and a solution $\tau_c$ of the inhomogeneous part.
\begin{align*}
\mathrm{w}_1(x)
&= 1 + \frac{2}{625}x^3 + O(x^4),
\\[4pt]
\mathrm{w}_2(x)
&= x - \frac{3}{10}x^2 + \frac{11}{75}x^3 + O(x^4),
\\[4pt]
\mathrm{w}_3(x)
&= x^2 - \frac{23}{30}x^3 + O(x^4),
\\[4pt]
\mathrm{w}_4(x) &= \log(x)\,\mathrm{w}_2(x)
+ \frac{9}{20}x^2 - \frac{169}{450}x^3 + O(x^4),\\
\tau_c(x)&=\frac{x^3}{160\sqrt5}\Big(1-\frac{13}{10}x+\frac{5421}{4000}x^2-\frac{3969373}{3000000}x^3+O(x^4)\Big).
\end{align*}

At the conifold, one can choose a flat coordinate that is invariant under monodromy around $z=5^{-5}$ and vanish at the conifold point. This implies that the numerator should include a vanishing period. The simplest possibility is:

\begin{equation}\label{eq:tc}
t_c(x)=\frac{\mathrm w_2}{\mathrm w_1} = x-\frac{3}{10}x^2+\frac{11}{75}x^3-\frac{9}{100}x^4 + O(x^5)
\end{equation}

which can be inverted as
\begin{equation}\label{mirror-coni}
 x(t_c) = t_c
+\frac{3}{10}t_c^2
+\frac{1}{30}t_c^3
+\frac{1}{200}t_c^4
+O(t_c^5)   
\end{equation}

It is important to notice that other flat coordinates can be considered and that the solutions $\mathrm w_i$ are not analytical continuations of the solutions around $z=0$. This differs from the situation around $z=0$ where there is one canonical choice of $q$. We follow~\cite{huang_topological_2009} and pick the same solutions as them for the ${\rm w}_i$. Here, ${\rm w}_2$ is the vanishing period and ${\rm w}_3$ is the solution with largest local exponent. $\rm w_1$ and $\rm w_4$ could have been chosen differently. That would change the expansions, but the gap condition (see equation~\eqref{eq:gap-intro} below and Section~\ref{sec:ambiguity}) would be preserved.

The extra period was taken from~\cite{walcher_extended_2009}, but it's the only solution of the inhomogeneous equation without $1,x,x^2$. 

The conifold expansion is going to be important when considering the higher genus amplitudes, as those satisfy a gap property, i.e., the higher genus amplitudes, can be written as:
\begin{equation}\label{eq:gap-intro}
\mathcal G_c^{\chi} = c_\chi t_c^{-\chi} + O(1)
\end{equation}

where $c_\chi=0$ if $\chi$ is odd. For details see Section~\ref{sec:ambiguity}.

\subsection{Orbifold expansion}

The orbifold expansion is around $z=\infty$ or $\psi=0$. Similarly to the conifold, we can choose a basis and obtain a coordinate $s$ around the orbifold given by the quotient of two solutions. Here, we also don't have a unique choice, since it depends on a choice of fifth root of the variable $u=\psi^5$.

We have:

\begin{align*}
    y_1(u) &= 5\sqrt{5}\,u^{1/5} + \frac{\sqrt{5}}{24}\,u^{6/5}
  + \frac{3}{56\sqrt{5}}\,u^{11/5} + \frac{14641}{611520\sqrt{5}}\,u^{16/5}
  + O\!\left(u^{21/5}\right)\\
    y_2(u) &= 5\sqrt{5}\,u^{2/5} + \frac{2\sqrt{5}}{9}\,u^{7/5}
  + \frac{2401}{7128\sqrt{5}}\,u^{12/5} + \frac{343}{2145\sqrt{5}}\,u^{17/5}
  + O\!\left(u^{22/5}\right)\\
    y_3(u) &= 5\sqrt{5}\,u^{3/5} + \frac{27\sqrt{5}}{56}\,u^{8/5}
  + \frac{64}{77\sqrt{5}}\,u^{13/5} + \frac{57122}{137445\sqrt{5}}\,u^{18/5}
  + O\!\left(u^{23/5}\right)\\
    y_4(u) &= 5\sqrt{5}\,u^{4/5} + \frac{16\sqrt{5}}{21}\,u^{9/5}
  + \frac{1458}{1001\sqrt{5}}\,u^{14/5} + \frac{9261}{12155\sqrt{5}}\,u^{19/5}
  + O\!\left(u^{24/5}\right)\\
\end{align*}

\text{and the extra period for the inhomogeneous equation:}

\begin{equation}\label{eq:tau-orb}
\tau_{orb}=-\frac{50\sqrt{5}}{3}\,u^{1/2}
-\frac{31250\sqrt{5}}{27027}\,u^{3/2}
-\frac{19531250\sqrt{5}}{52055003}\,u^{5/2}
+O\!\left(u^{7/2}\right)
\end{equation}

Again, we use the flat coordinate 

\begin{equation}\label{eq:s-orb}
s(u) =\frac{y_2(u)}{y_1(u)} = 
 u^{1/5}+\frac{13}{360}\,u^{6/5}+\frac{110069}{9979200}\,u^{11/5}+\frac{52200923}{9906624000}\,u^{16/5}+O\!\left(u^{21/5}\right),
\end{equation}

and obtain its inverse

\begin{equation}\label{eq:u-orb}
u(s)=s^{5}-\frac{13}{72}\,s^{10}-\frac{1193}{399168}\,s^{15}+O\!\left(s^{20}\right).
\end{equation}

By using the inverse, we can expand our objects around $\psi=0$ and study its properties. 

We are going to see in the next section, that the data from the Picard-Fuchs equation can be encoded in the Gauss-Manin connection of the family of mirror quintics. The construction also takes the variation of Hodge structure into account and allows us to study special geometry in an algebraic framework.

\subsection{Enhanced Mirror Quintics}\label{sec-gmcd}

Having recalled the basic genus zero setup and the mirror map, let us recall the basic framework of Gauss-Manin connection in Disguise which will be used in this paper. Recall that, although this paper is focused on the mirror quintic, the ideas can by applied more generally. 

Let $X\to S$ be a family of projective varieties and $Y\subset X$ be a closed subvariety. We can consider the relative algebraic de Rham cohomology ${\rm H}^n(X_s,Y_s)$, where $n$ is the dimension of $X$. Those are fibers of the cohomology bundle over $S$. Such cohomology has a natural mixed Hodge structure coming from the long exact sequence of the pair. When $Y_s$ is empty, this is just the pure Hodge structure of $X_s$. We want to consider a space parametrizing $S$ and basis of ${\rm H}^n(X_s,Y_s)$, i.e., triples $(X_s,Y_s,\alpha_s)$, that respect the mixed Hodge structure and that have constant intersection product. Recall that we have a natural pairing on the middle cohomology given by integration. We proceed to explain how it goes for the mirror quintic. The reader should keep in mind that $X\to S$ is going the mirror quintic family and the subvariety $Y$ above is going to be a pair of curves $C_{\pm}$. We call the triples $(X_s,Y_s,\alpha_s)$ that we parametrize {\it relatively enhanced mirror quintics}.

Our goal here is to recall the definition of an enhanced mirror quintic and its real counterpart (relatively enhanced mirror quintics). We also describe the moduli spaces of (relatively) enhanced mirror quintics, but leave the proofs and technical details out. The definitions here were first made in~\cite{movasati_modular-type_2015} and~\cite{espreafico_gauss-manin_2026}, where the reader can also find more details.

 \begin{definition}\label{def:mirror-quintic}
     A mirror quintic is a resolution of singularities of an element of the Dwork family:
     
\begin{equation}\label{eq:dwork}
Y_{\psi}:x_0^5+x_1^5+x_2^5+x_3^5+x_4^5-5\psi x_0x_1x_2x_3x_4=0
\end{equation}

    in the quotient of $\mathbb P^4$ by the action the group $\mathbb Z_5^3$ given by:
    \begin{align*}
e_1 [x_0, x_1, x_2, x_3, x_4] &= [\zeta x_0, x_1, x_2, x_3, \zeta^{-1} x_4] \\
e_2 [x_0, x_1, x_2, x_3, x_4] &= [x_0, \zeta x_1, x_2, x_3, \zeta^{-1} x_4] \\
e_3 [x_0, x_1, x_2, x_3, x_4] &= [x_0, x_1, \zeta x_2, x_3, \zeta^{-1} x_4].
\end{align*}
where $e_i$'s are the generators of the group and $\zeta^5=1$.
 \end{definition}

Mirror quintics have the Calabi-Yau properties and they have Hodge diamond mirror to the quintic threefold. We won't need the constructions explicitly, as we will only use properties of the variation of Hodge structure of this family. For details, see~\cite{greene_duality_1990} or~\cite[Section 2.2]{cox_mirror_1999}.

In order to obtain Eisenstein type functions, we define enhanced mirror quintics, by considering pair of mirror quintics and a basis of the third de Rham cohomology that is compatible with the Hodge structure.

In order to obtain information about real enumerative geometry of the quintic, we need to consider a pair of conics in each mirror quintic, known as Deligne conics. They are defined by the equations
\begin{equation}
C_{\pm}=\left\{ x_{0}+x_{1}=0, x_{2}+x_{3}=0, {x_4}^{2} \pm \sqrt{5\psi}x_{1} x_{3}=0\right\}.
\end{equation}

We denote by $\langle\cdot,\cdot\rangle$ intersection form normalized as in~\cite{movasati_modular-type_2015,alim_gaussmanin_2016}:
\begin{equation}\label{eq:pairing}
\langle\omega_1,\omega_2\rangle:=\frac{1}{(2\pi i)^3}\int_{Y_\psi}\omega_1\wedge\omega_2,\qquad \omega_1,\omega_2\in H^3_{dR}(Y_\psi),
\end{equation}
extended to $H^3_{dR}(Y_\psi,C_\pm)$ through the natural map $H^3_{dR}(Y_\psi,C_\pm)\to H^3_{dR}(Y_\psi)$; in particular the class $\alpha_0\in W_2$ of Definition~\ref{relat-enha-mq} pairs trivially with every class, which is the first row and column of $\Phi$.

Those conics are roughly corresponding to the real Lagrangian via mirror symmetry and therefore we have to consider the relative cohomology $H^3(Y_{\psi},C_+\cup C_-)=:H^3(Y_{\psi},C_\pm)$. We arrive at the definition of relatively enhanced mirror quintics, which are pairs of mirror quintics and a basis of $H^3(Y_{\psi},C_\pm)$ compatible with the natural \textit{mixed} Hodge structure that exists on the relative cohomology. We arrive at the following:

\begin{definition}[cf.~\cite{espreafico_gauss-manin_2026}]%
\label{relat-enha-mq}
    A relatively enhanced mirror quintic is a triple \[(Y_\psi,[\alpha_0,\dots,\alpha_4]),\] where $Y_\psi$ is a mirror quintic, and $[\alpha_0,\dots,\alpha_4]$ is a basis of $H^3_{dR}(Y_{\psi},C_\pm)$ satisfying the following properties:
\begin{multicols}{2}
\begin{enumerate}
\item $\alpha_i\in F^{4-i}\setminus F^{5-i},\quad i >0$;

\item $[\langle\alpha_i,\alpha_j\rangle] =\Phi$;

\item $\alpha_0\in F^1\setminus F^2$;

\item $\alpha_0\in W_2$;

\item $\int_{\delta_0} \alpha_0 = 1$;

\item $\alpha_i\in W_3\setminus W_2,\quad i>0$.
\end{enumerate}
\end{multicols}
where $\Phi=\left[\begin{array}{c
cccc}
0 & 0 & 0 &0 &0\\
0&0 & 0 & 0 & 1 \\
0&0 & 0 & 1 & 0 \\
0&0 & -1 & 0 & 0 \\
0&-1 & 0 & 0 & 0
\end{array}\right]$.

\end{definition}

\begin{remark}\label{rem:enhanced-Mov}
    To obtain the definition of enhanced mirror quintic from~\cite{movasati_modular-type_2015}, it suffices to disregard properties 3 to 6, remove $\alpha_0$ (since $\dim H^3(Y_\psi)=4$) and remove the first row and column of the matrix $\Phi$.
\end{remark}

With the definitions in hand, we can consider the space $\rT$ of all relatively enhanced mirror quintics. We also denote by $\cT$ the space of enhanced mirror quintics, i.e., enhanced only with a basis of the absolute cohomology (the one considered in~\cite{movasati_modular-type_2015} and~\cite{alim_gaussmanin_2016}). 

The space $\rT$ has a natural action of the algebraic group
\begin{equation}
{\sf H}:=
\left\{\left(\begin{array}{ccccc}
1 & 0 &0 &h_1 &h_2\\
0 &g_{11} & g_{12} & g_{13} & g_{14} \\
0 & 0 & g_{22} & g_{23} & g_{24} \\
0 & 0 & 0 & g_{33}& g_{34} \\
0 &0 & 0 & 0 & g_{44}\\
\end{array}\right), \quad h_k,g_{ij}\in\mathbb C\quad \begin{array}{c}
g_{11}g_{44}=1,\\
g_{22}g_{33} = 1,\\
g_{12}g_{44} + g_{22}g_{34} = 0,\\
g_{13}g_{44}+g_{23}g_{34}-g_{24}g_{33} = 0.
\end{array}\right\}
\end{equation}

$\sf H$ consists of the changes of basis that preserve properties in~\ref{relat-enha-mq} and it acts naturally on an element of $\rT$ by changing the basis $\alpha$. Notice that it preserve the mixed Hodge structure and the intersection product. For $\cT$, we can define a group $\sf G$ acting in the same way by removing the first row and first column. A very important object is going to be the lie algebra of the group $\sf H$, given by:

\begin{equation}\label{eq:def-h}
\mathfrak h:=
\left\{\left(\begin{array}{ccccc}
0 & 0 &0 &h_1 &h_2\\
0 &g_{11} & g_{12} & g_{13} & g_{14} \\
0 & 0 & g_{22} & g_{23} & g_{24} \\
0 & 0 & 0 & g_{33}& g_{34} \\
0 &0 & 0 & 0 & g_{44}\\
\end{array}\right), \quad h_k,g_{ij}\in\mathbb C\quad \begin{array}{c}
g_{11}+g_{44}=0,\\
g_{22}+g_{33} = 0,\\
g_{12}+g_{34} = 0,\\
g_{13}-g_{24} = 0.
\end{array}\right\}
\end{equation}
Again, the version for $\cT$ and $\sf G$ is given by removing the first row and column. 

The last ingredient that will be important to us is the Gauss-Manin connection. It simply computes the derivations of the forms in the de Rham cohomology with respect to the parameters on the base of the family. It is defined purely algebraically (see~\cite{katz_differentiation_1968} for the definition for absolute cohomology and~\cite{espreafico_modularity_2024} for the relative version). It is a map:

\begin{equation}\label{eq:GM-connection}
\nabla:{\rm H}^3(Y,C_{\pm}/\rT)\to{\rm H}^3(Y,C_{\pm}/\rT)\otimes\Omega^1_{\rT}
\end{equation}

where ${\rm H}^3(Y,C_{\pm}/\rT)$ denotes the algebraic relative de Rham cohomology with respect to the family $\rT$. Given a vector field in $\rT$, one can define a matrix with entries in $\mathcal O_{\rT}$, denoted by $\nabla_R=A_R$. It is simply given by writing the connection in the basis $\alpha$ which is encoded in each point of $\rT$. Notice that this is well defined in the whole space $\rT$, since the basis is part of the data for each point of $\rT$.

The Gauss-Manin connection encodes the differential data that is encoded in the classic Picard-Fuchs equation for periods. It gives rise to a special vector field, that encodes the Yukawa coupling and the disk potential. Such vector field provides a derivation on $\mathcal O_{\rT}$ and allows us to consider $\mathcal O_{\rT}$ as a generalization of the algebra of modular forms. The moduli space and the vector field were computed explicitly in the previous work~\cite{espreafico_gauss-manin_2026} (and~\cite{movasati_modular-type_2015} for $\cT$).

\begin{proposition}\label{prop:R0} (cf.~\cite{espreafico_gauss-manin_2026}) $\rT$ is a 9-dimensional quasi-affine variety given by
\begin{equation}\label{eq: def-rT}
    \rT=\mathbb C\left[s_0,\dots s_8,\frac1{s_0s_4s_5(s_0^{10}-s_4^{10})}\right]
\end{equation}
and there is a unique vector field ${\sf R}_0$, and unique regular functions $\sf Y$ and $\sf F$ on $\rT$ such that 

\begin{equation}\label{eq:GM-R0}
\nabla_{{\sf R}_0} = \begin{bmatrix}0 &0&0&0&0\\
0&0&1&0&0\\
F&0&0&Y&0
\\
0&0&0&0&-1
\\
0&0&0&0&0\end{bmatrix}
\end{equation}

\begin{equation}\label{eq:R0}
{\sf R}_0:\quad\left\{\begin{aligned}
\dot s_0 &= \frac{6\cdot5^{4}s_0^{10}+s_0^{2}s_3-5^{4}s_4^{10}}{2s_0s_5}\\
\dot s_1 &= \frac{-5^{8}s_0^{12}+5^{5}s_0^{8}s_1+5^{8}s_0^{2}s_4^{10}+s_1s_3}{s_5}\\
\dot s_2 &= \frac{-3\cdot5^{9}s_0^{14}-5^{4}s_0^{10}s_1+2\cdot5^{5}s_0^{8}s_2
   +3\cdot5^{9}s_0^{4}s_4^{10}+5^{4}s_1s_4^{10}+2s_2s_3}{s_5}\\
\dot s_3 &= \frac{-5^{10}s_0^{16}-5^{4}s_0^{10}s_2+3\cdot5^{5}s_0^{8}s_3
   +5^{10}s_0^{6}s_4^{10}+5^{4}s_2s_4^{10}+3s_3^{2}}{s_5}\\
\dot s_4 &= \frac{s_4\left(5^{5}s_0^{8}+s_3\right)}{2s_5}\\
\dot s_5 &= \frac{3\cdot5^{5}s_0^{8}s_5+2s_3s_5-5^{4}\left(s_0^{10}-s_4^{10}\right)s_6}{s_5}\\
\dot s_6 &= \frac{-5^{5}s_0^{6}s_5-2s_2s_5+3\cdot5^{5}s_0^{8}s_6+3s_3s_6}{s_5}\\
\dot s_7 &= -s_8\\
\dot s_8 &= \frac{3\cdot5^{21/2}s_4^{5}\left(s_0^{10}-s_4^{10}\right)}{8s_0^{5}s_5}
\end{aligned}\right.
\end{equation}

\begin{equation}\label{eq:YF}
{\sf Y}=\frac{5^{8}\left(s_0^{10}-s_4^{10}\right)^{2}}{s_5^{3}},\quad {\sf F}=-{\sf Y}s_7
\end{equation}

\end{proposition}

\begin{proposition}\label{prop:expansions}
    Writing ${\sf R}_0$ as a differential equation in $s_i(q)$ with differential operator $5q\frac{d}{dq}$ and solving with the initial conditions 
\begin{equation}\label{eq:init-LCS}
    s_0 = \frac1{\sqrt{5}} + 12\sqrt{5}q + O(q^2)\quad    s_4 = q^{1/10}+O(q^{11/10})
\end{equation}
 gives us the expansions around $z=0$. Note that we need to allow powers $q^{1/2}$ for $s_7$ and $s_8$ and $q^{1/10}$ for $s_4$. Writing $z=\frac{s_4^{10}}{5^5s_0^{10}}$ and expanding gives us the expression of $z$ in terms of $q$ via mirror map (this is how we can deduce the first coefficients in the first place).

    On the other hand, using the initial conditions 
\begin{equation}\label{eq:init-coni}
s_0 = \frac1{\sqrt 5}+\frac1{625\sqrt{5}}t_c^3 + O(t_c^4)\qquad s_5 = t_c - \frac{3}{10}t_c^2 + O(t_c^3)
\end{equation}
 and solving for $5\frac{d}{dt_c}$ we get the expansions around $z=5^{-5}$.  In particular, writing $x=\frac{1}{5^5z}-1=\frac{s_0^{10}}{s_4^{10}}-1$ gives the expression~\eqref{mirror-coni}.
\end{proposition}

We delay the proof of Proposition~\ref{prop:expansions} until Corollary~\ref{cor:tau0-mirror-map}. 

\begin{remark}\label{rem:orbifold}
    We can also do the same for the orbifold point, but it will not be necessary in our computations.
\end{remark}

\begin{remark}\label{rem:integral-curve}
    We will see in Section~\ref{sec: periods}, that the integral curve of ${\sf R}_0$ is a one dimensional space that correspond to a period domain through the period map.
\end{remark}

\section[Extended Holomorphic Anomaly Equation and Alim-Länge Polynomiality]{Extended Holomorphic Anomaly Equation and\\ Alim-Länge Polynomiality}\label{sec: Alim-Lange}

In this section we quickly recall the special geometry formalism linked to the moduli space of mirror quintics, the extended holomorphic anomaly equations proposed by Walcher~\cite{walcher_evidence_2009} which is satisfied by the generating functions $\mathcal G^\chi$ (indexed by Euler characteristic) and how the solutions of such equation can be expressed in terms of special sections of vector bundles as done in~\cite{alim_polynomial_2007}.

    Let $\mathcal M$ be the complex structure moduli. In the case of the mirror quintic this is simply $\mathbb P^1-\{0,5^{-5},\infty\}$. In general, we have more dimensions, but here we focus on the one dimensional case with a variable $z$. The geometry of $\mathcal M$ is described by looking at the third cohomology bundle $\mathcal H^3(Y_z,\mathbb C)$, with the natural intersection form. Such bundle can be decomposed as
    
\begin{equation}\label{eq:hodge-decomp}
\mathcal L\oplus(\mathcal L\otimes T\mathcal M)\oplus\overline{\mathcal L\oplus(\mathcal L\otimes T\mathcal M)},
\end{equation}

    where $\mathcal L$ is generated by the holomorphic three form $\Omega$ and $T\mathcal M$ is the tangent bundle. 

    There is a natural metric defined from the Kähler potential $K:=-\log\langle\Omega,\bar\Omega\rangle$ as $G_{z\bar z}:=\partial_{\bar z}\partial_zK$. Such metric gives rise to the Levi-Civita connection with Christoffel symbol $\Gamma^{z}_{z z}=G_{z\bar z}^{-1}\partial_zG_{z\bar z}$.

    The fundamental object is the Yukawa coupling $C_{zzz}:=-\langle\Omega,\partial_z^3\Omega\rangle$, which is a section of $\Gamma(\mathcal L^{2}\otimes{\rm Sym}^3T^*\mathcal M)$, that gives the following relation
\begin{equation}\label{eq:special-geometry}
     \partial_{\bar z}\Gamma_{zz}^z  = 2G_{z\bar z} -C_{zzz}\overline{C_{zzz}} e^{2K} G_{z\bar z}^{-2}
\end{equation}
We fix the normalization where $C_{zzz}= \frac{5}{z^3(1-5^5z)}$ following~\cite{alim_polynomial_2007}. This is not the same normalization that appears in~\cite{espreafico_gauss-manin_2026,movasati_modular-type_2015}. This is addressed in~\ref{sec:dict}, when we compare both frameworks. The factors of five appearing there are caused by this difference in normalization.

   There is a covariant derivative defined on the bundles $\mathcal L^n\otimes {\rm Sym}^m T\mathcal M$ from the Levi Civita connection and the derivative of $K$, denoted by $D_z$. 
   
\begin{equation}\label{eq:Dz}
D_z=\partial_z +m\Gamma_{zz}^z+n \partial_zK
\end{equation}

In the case of the real/open theory, one has to consider an extra section, denoted by $\Delta_{zz}\in\Gamma({\rm Sym}^2T^*\mathcal M\otimes\mathcal L)$, related to the extra cohomology class that appears in relative cohomology $H^3(Y,C_{\pm})$ as in the introduction. It is given by $\langle\Omega,\partial_z^2\nu\rangle$, where $\nu$ is lift of the normal function associated to integration over the extra relative cycle $C_{+}$ connecting $C_{-}$. The function $\Delta_{zz}$ satisfies the equation 
\begin{equation}\label{eq:Deltazz-difeq}
    \partial_{\bar{z}} \Delta_{zz} = C_{zzz} e^{K} G_{z\bar{z}}^{-1} \overline{\Delta_{zz}},
\end{equation}
and it can be determined from the solution to the nonhomogeneous Picard-Fuchs equation $\tau$ and the covariant derivative.
\begin{equation}\label{eq:deltazz-def}
    \Delta_{zz} = i\left(D_z D_z \tau - \frac{e^{K} C_{zzz}}{G_{z\bar{z}}} \overline{D}_{\bar{z}} \overline{\tau}\right) ,
\end{equation}

The topological string amplitudes are also sections of powers of $\mathcal L$. They correspond, through mirror symmetry, to the generating functions of Gromov-Witten invariants. In our case, of the quintic threefold. The functions $\mathcal F^{g}\in \Gamma(\mathcal L^{2-2g})$ are the closed string amplitudes and correspond to complex Gromov-Witten invariants of genus $g$ (i.e., virtual counts of $J$-holomorphic curves of genus $g$).  The string amplitudes $\mathcal G^{\chi}\in\Gamma(\mathcal L^{-\chi})$ corresponding to a sum of real Gromov-Witten invariants and complex Gromov-Witten invariants of fixed Euler characteristic (i.e., virtual counts of $J-$holomorphic curves of Euler characteristic $-\chi$ with an involution compatible with complex conjugation).  Notice that the source of these $J$-holomorphic curves can be orientable or not. $\mathcal G^{\chi}$ is related to other amplitudes in the following way:

\begin{equation}\label{eq:G-decomp}
\mathcal G^{\chi} =\frac1{2^{\chi/2+1}}\left(\mathcal F^{g_{\chi}} + \sum_{2g-2+h=\chi}(\mathcal K^{g,h} +\mathcal F^{g,h}) + \sum_{2g-1+h=\chi}\mathcal R^{g,h}\right)
\end{equation}

where $g_\chi=\frac{\chi}2+1$: this term appears only for even $\chi$. Here $\mathcal F$ correspond to orientable surfaces, $\mathcal K$, to nonorientable surfaces with an even number of crosscaps and $\mathcal R$ to nonorientable surfaces with an odd number of crosscaps. The real Gromov-Witten invariants $N^{\mathbb R}_{g,d}$ are encoded in the expansion of the open part (including the non orientable) for $g=\chi+1$. In other words, $\mathcal G^{\chi}$ encodes the complex invariants of $g_\chi=\chi/2+1$ and real invariants of $g = \chi+1$. We won't get into details regarding $\mathcal{F,K,R}$, just consider the full generating function $\mathcal G$. 

Throughout this paper, we adopt the change of coordinates ${\mathcal G}^{\chi} \to 2^{\chi/2}\mathcal G^{\chi}$ to simplify computations. This will slightly change the anomaly equations, eliminating the traditional factor of $\frac12$ that appears in most papers and in the complex case. In terms of $\mathcal F$, $\mathcal K$ and $\mathcal R$, we have:
\begin{equation}\label{eq:G-decomp-norm}
\mathcal G^{\chi} =\frac12\left(\mathcal F^{g_{\chi}} + \sum_{2g-2+h=\chi}(\mathcal K^{g,h} +\mathcal F^{g,h}) + \sum_{2g-1+h=\chi}\mathcal R^{g,h}\right)
\end{equation}
The correspondence between these amplitudes and Gromov-Witten theories appears once we consider the holomorphic $q-$expansions of these functions in the large complex structure limit, i.e., around $z=0$. 

\begin{equation}\label{eq:hol-limit}
\mathcal F^g_A=\lim_{\bar z\to0}w_0^{2g-2}\mathcal F^{g}\qquad \mathcal G^\chi_A =\lim_{\bar z\to0}w_0^\chi \mathcal G^{\chi}
\end{equation}

Here, $w_0$ is the holomorphic solution to the Picard-Fuchs equation for periods of the 3-form $\Omega$. It correspond to the holomorphic limit of the Kähler parameter $e^{-K}$. Multiplying by the period transforms sections into regular functions on $\mathcal M$ that can be expanded in the coordinate $q$ (see section~\ref{pic-fuchs}).

The coefficients of $q^d$, according to Mirror Symmetry, correspond to the Gromov-Witten invariants of degree $d$ (for the quintic, this has been proved following works of Guo and collaborators see~\cite{ guo_structure_2018,chang_polynomial_2021,chang_bcov_2026, chang_conifold_2026}). Such generating functions have a so called BPS expansion, which accounts for multiple covers that appear when computing virtual counts of stable maps.

\begin{equation}\label{eq:BPS-complex}
\sum_{g}\mathcal F^g_A\lambda^{2g-2} = \sum_{d,g} N_{g,d}^{\mathbb C}q^d\lambda^{2g-2} = \sum_{g,k,d} n_{g,d}^{\mathbb C}\frac1k\left(2\sin\left(\frac{k\lambda}{2}\right)\right)^{2g-2}q^{kd}
\end{equation}

\begin{equation}\label{eq:BPS-real}
\sum_{\chi}\left(i\lambda\right)^\chi\left(\mathcal G^{\chi}_A-\frac12\mathcal F^{g_\chi}_A\right) = 
\sum_{\substack{\chi \equiv d \bmod 2 \\ k\ \mathrm{odd}}}
n_{\chi+1,d}^{\mathbb R}
\frac{1}{k}
\left(2\sinh\frac{\lambda k}{2}\right)^\chi
q^{kd/2}
\end{equation}
\begin{remark}\label{rem:GV-integrality}
    The invariants $n_{g,d}$ are known as Gopakumar-Vafa (GV) invariants and they have the remarkable property of being integers. For the complex invariants, this was proved in the $A$-model, i.e., using properties of Gromov-Witten invariants in symplectic geometry by Ionel and Parker~\cite{ionel_gopakumar-vafa_2018}. For the real invariants, there have been great advances made by Georgieva and Ionel in proving integrality~\cite{georgieva_splitting_2021,georgieva_klein_2021}. 

    Integrality provides the main check for our computations. Since the expression above does not give integers for any set of $N_{g,d}$, obtaining integer GV invariants is strong evidence that the invariants computed are correct.
\end{remark}
\begin{remark}\label{rem:real-genus}
    Note that the real invariants for $\chi$ are invariants of genus $\chi+1$. This is explained by the fact that a surface with boundary (orientable or not) of Euler characteristic $-\chi$ can be seen as a quotient of a closed surface of Euler characteristic $-2\chi$ by an involution. By definition, the genus of the real invariants is the genus of the corresponding closed Riemann surface. In this case,  $2-2g=-2\chi$, which gives $g=\chi+1$.
\end{remark}

The amplitudes satisfy a recursive differential equation known as the holomorphic anomaly equation, introduced in the closed case in the BCOV papers~\cite{bershadsky_holomorphic_1993,bershadsky_kodaira-spencer_1994}. For the real case, an extension was coined by Walcher in~\cite{walcher_extended_2009} for generating functions of oriented counts $\mathcal F^{g,h}$. Later he considered a version for $\mathcal G^{\chi}$ in~\cite{walcher_evidence_2009}, which ended up giving the same equation. Fix the notation
\begin{equation}\label{eq:Cbar-Deltabar}
C_{\bar{z}}^{zz} = \overline{C_{zzz}} e^{2K} G_{z\bar{z}}^{-2}, \quad \Delta_{\bar{z}}^{z} = \overline{\Delta_{zz}} e^{K} G_{z\bar{z}}^{-1} .
\end{equation}

The holomorphic anomaly equation for the closed topological string amplitudes reads, for $g\geq2$

\begin{equation}\label{HAE-cl}
    \partial_{\bar{z}} {\mathcal F}^{g} = \frac{1}{2} \sum_{\smallmatrix g_1 + g_2 = g\\g_1,g_2>0\endsmallmatrix} C_{\bar{z}}^{zz} {\mathcal F}_1^{g_1}{\mathcal F}_1^{g_2} + \frac{1}{2} C_{\bar{z}}^{zz} {\mathcal F}_{2}^{g-1} .
\end{equation}
where the sub-indices are covariant derivatives, i.e., $\mathcal F^g_n=(D_z)^n\mathcal F^{g}$. 

The extended version has an extra term related to $\Delta_{zz}$, for $\chi\geq1$:
\begin{equation}\label{HAL-real}
    \partial_{\bar{z}} {\mathcal G}^{\chi} =  \sum_{{\smallmatrix \chi_1 + \chi_2 = \chi - 2\\ \chi_1,\chi_2\geq0\endsmallmatrix}} C_{\bar{z}}^{zz} {\mathcal G}_1^{\chi_1}{\mathcal G}_1^{\chi_2} + C_{\bar{z}}^{zz} {\mathcal G}_{2}^{\chi-2} + \Delta_{\bar{z}}^{z} {\mathcal G}_1^{\chi-1} .
\end{equation}

\begin{remark}\label{rem:normalization}
    As mentioned, the equations for $\mathcal G^{\chi}$ do not have the factor $\frac12$ because of the normalization chosen. Indeed, without the normalization, as we see in~\cite{walcher_evidence_2009}, there is a factor $\frac12$, but multiplying the equation by $2^{\chi/2}$ will cancel the factor $\frac12$. Note that our $\Delta_{zz}$ differs from Walcher's by a factor of $\sqrt2$, which is more suitable for the normalized equation and is important so the equation as above is actually satisfied. In his paper, our $\Delta_{zz}$ correspond to equation~\cite[Eq (5.7)]{walcher_evidence_2009} and not to the general $\mathcal W$ that appears on the derivation of the anomaly equation and has the factor $\sqrt2$.
\end{remark}

We also fix the first few functions $\mathcal F^0_3$ and $\mathcal F^1_1$ for the complex case and $\mathcal G^{-2}_3$, $\mathcal G^{-1}_2$ and $\mathcal G^{0}_1$. For $g=0$ or $\chi=-2$, we have the Yukawa coupling $\mathcal G^{-2}_3=\frac12C_{zzz}$. $\mathcal G^{-1}_2 =\frac12\Delta_{zz}$ and $\mathcal G^{0}_1$ and $\mathcal F^{1}_1$ satisfy:

\begin{equation}\label{eq:G01}
\partial_{\bar z}\mathcal G^{0}_1= \frac12\Delta^{z}_{\bar z}\Delta_{zz} -\frac{e(X)}{48}G_{z\bar z}+\frac12C_{zzz}C^{zz}_{\bar z}
\end{equation}

\begin{equation}\label{eq:F11}
\partial_{\bar z}\mathcal F^{1}_1= -\left(\frac{e(X)}{24}-1\right)G_{z\bar z}+\frac12C_{zzz}C^{zz}_{\bar z}
\end{equation}

where $e(X)=-200$ is the Euler characteristic of the quintic.

These equations determine the amplitudes $F^g$ and $G^{\chi}$ up to an holomorphic term, which is called holomorphic ambiguity, including for $\chi=0$. The first two ambiguities can be fixed immediately from the information about the low degree invariants, which is available for genus 1. For higher genus, we need to use different properties (see Section~\ref{sec:ambiguity}). These equations can be solved using Feynman diagrams (see~\cite{bershadsky_holomorphic_1993,bershadsky_kodaira-spencer_1994},\cite{walcher_extended_2009} and~\cite{konishi_solutions_2007}). The idea is that the solutions can be written as certain sums over graphs with vertices representing Riemann surfaces of certain genus. There are different types of edges correspond to certain objects called \textit{propagators}. In the real case, some edges do not have ends (they represent boundary components) which correspond to \textit{terminators}. These functions allow us to compute a potential for a specific graph as a product over edges and then sum over graphs account for automorphisms. We do not go into details here.

The propagators are sections defined by the equations

\begin{equation}\label{eq:propagators}
\partial_{\bar{z}} S^{zz} = C_{\bar{z}}^{zz}, \quad \partial_{\bar{z}} S^{z} = S^{zz} G_{z\bar{z}}, \quad \partial_{\bar{z}} S = S^{z} G_{z\bar{z}},
\end{equation}

The terminators are defined by the following equations

\begin{equation}\label{eq:terminators}
\partial_{\bar{z}} \Delta^{z} = -\Delta_{\bar{z}}^{z}, \quad \partial_{\bar{z}} \Delta = \Delta^{z} G_{z\bar{z}}.
\end{equation}

Notice that these definitions only define the propagators up to an holomorphic term. By using the known expansions for the $\chi\leq0$ (genus 0 and 1), all these can be fixed. See the last section of~\cite{alim_polynomial_2007}. With all these in hands, we can recall the main result of Yamaguchi and Yau~\cite{yamaguchi_topological_2004} for the closed amplitudes for the quintic, later extended by Alim and Länge~\cite{alim_polynomial_2007} that says that the nonholomorphic part of the solutions of the holomorphic anomaly equation are polynomials in $S^{zz},S^{z},S,\Delta^z,\Delta$, $\partial_zK$. This is proved by showing that the covariant derivative of the propagators can be written in terms of the propagators (plus holomorphic ambiguity) and then rewriting the holomorphic anomaly equation in terms of these generators (see equations 2.23 to 2.25 of~\cite{alim_polynomial_2007}). Here, we present the form adapted for $\mathcal G^{\chi}$ instead of the $\mathcal F^{g,h}$.

\begin{proposition}\label{al-HAE} Walcher's extended holomorphic anomaly equation can be written as 
\begin{align}
\frac{\partial \mathcal G^{\chi}}{\partial S^{zz}} &= \sum_{\substack{\chi_1+\chi_2=\chi-2\\ \chi_1,\chi_2\geq0}}  \mathcal G_1^{\chi_1} \mathcal G_1^{\chi_2} + \mathcal G_2^{\chi-2}, \\[4pt]
\frac{\partial \mathcal G^{\chi}}{\partial \Delta^{z}} &= -\mathcal G_1^{\chi-1}, \\[4pt]
0 &= \frac{\partial \mathcal G^{\chi}}{\partial K_z} + S^z \frac{\partial \mathcal G^{\chi}}{\partial S} + S^{zz} \frac{\partial \mathcal G^{\chi}}{\partial S^z} + \Delta^z \frac{\partial \mathcal G^{\chi}}{\partial \Delta} .
\end{align}
\end{proposition}
\begin{proof}
    The proof is basically the same as what appears in~\cite{alim_polynomial_2007}. The main difference is the factor $\frac12$ that disappears from the first equation. This happens because of the normalization we use for $\mathcal G^{\chi}$, as explained before. For completion, we reproduce their idea here.

  Since $\mathcal G^{\chi}$ is a polynomial in the propagators, terminator and $K_z$, we have:
    
\begin{equation}\label{eq:dbarG-chain}
\partial_{\bar z}\mathcal G^{\chi} = \partial_{\bar z}S^{zz}\frac{\partial \mathcal G^{\chi}}{\partial S^{zz}} + \partial_{\bar z}S^{z}\frac{\partial \mathcal G^{\chi}}{\partial S^{z}} +\partial_{\bar z}S\frac{\partial \mathcal G^{\chi}}{\partial S} +\partial_{\bar z}\Delta^{z}\frac{\partial \mathcal G^{\chi}}{\partial \Delta^z}+\partial_{\bar z}\Delta\frac{\partial \mathcal G^{\chi}}{\partial \Delta}+\partial_{\bar z}K_z\frac{\partial \mathcal G^{\chi}}{\partial K_z}
\end{equation}

    Using the relations above and the fact that $\partial_{\bar z}\partial_zK=G_{z\bar z}$, we have:
    
\begin{equation}\label{eq:dbarG-grouped}
\partial_{\bar z}\mathcal G^{\chi} = C_{\bar z}^{zz}\frac{\partial \mathcal G^{\chi}}{\partial S^{zz}} - \Delta^z_{\bar z}\frac{\partial \mathcal G^{\chi}}{\partial \Delta^z} + G_{z\bar z}\left(S^{zz}\frac{\partial \mathcal G^{\chi}}{\partial S^{z}} +S^{z}\frac{\partial \mathcal G^{\chi}}{\partial S} + \Delta^z\frac{\partial \mathcal G^{\chi}}{\partial \Delta}+\frac{\partial \mathcal G^{\chi}}{\partial K_z}\right)
\end{equation}

    We finish by writing the anomaly equation on the LHS and comparing the terms multiplying $G_{z\bar z}$, $C_{\bar z}^{zz}$ and $\Delta^{z}_{\bar z}$. 
\end{proof}

\section[The Algebraic Holomorphic Anomaly Equation and Period Manipulation]{The Algebraic Holomorphic Anomaly Equation\\ and Period Manipulation}\label{sec: algebraic Hae}

In this section, our goal is to rewrite the extended anomaly equation in terms of the vector fields ${\sf R}_i$ and the variables $s_i$ from the introduction, i.e., prove Theorem~\ref{alg-HAE-GMCD-Int}. In order to relate the parameters from the moduli space $\rT$ to the propagator and terminators, our idea is to consider a real map $\sigma:\mathcal M\to \rT$, with a choice of basis at each point. This was considered in~\cite{alim_gaussmanin_2016} for the closed case, although never written in such explicit manner as we do in~\ref{sec:dict}. We establish a \textit{dictionary} relating each generator of the Alim-Länge polynomial ring to a regular function on $\rT$.

After having established the dictionary, it is natural to think about the expansions around the special points of $\mathcal M$. In order to understand that, we use the period map and the algebraic group action described in~\ref{sec-gmcd}. We show that the expansions around the special points can be computed by solving the differential system corresponding to $\sf R_0$, as stated in Proposition~\ref{prop:expansions}. We also show that the Lie algebra $\mathfrak h$ from~\eqref{eq:def-h} gives rise to the eight vector fields $\mathbb R_i$ that are involved in our version of the holomorphic anomaly equation from Theorem~\ref{alg-HAE-GMCD-Int}. In particular, the higher genus functions $G^\chi$ are regular functions in $\mathcal O_{\rT}$.

\subsection{The dictionary}\label{sec:dict}

Let $\sigma:\mathcal M\to \rT$ be the map $z\mapsto (Y_z,C_{\pm},\{\alpha_0,\allowbreak -i5^{-3}e^{K}\Omega,\allowbreak -i5^{-2}e^KG_{z\bar z}^{-1}(\partial_z\Omega+K_z\Omega),\allowbreak -i5^{2}(\partial_{\bar z}\bar\Omega+K_{\bar z}\bar\Omega),\allowbreak i5^{3}\bar \Omega\}$), where $\alpha_0$ is the cocycle associated to the normal function $\int_\delta$, as we describe in section~\ref{sec-gmcd}.

To see that this map is well defined, first notice that the forms respect the Hodge filtration. $\Omega$ is a (3,0) form, $D_z\Omega$ is a sum of a (3,0) form with a (2,1) form etc. To see the intersection product does not depend on $z$ and is equal to the matrix $\Phi$ (see~\ref{relat-enha-mq}), we have:

\begin{equation}\label{eq:sigma-pair1}
\langle -i5^{-3}e^{K}\Omega,i5^{3}\bar\Omega\rangle = e^K\langle\Omega,\bar\Omega\rangle = e^Ke^{-K} = 1
\end{equation}

\begin{equation}\label{eq:sigma-pair2a}
\langle -i5^{-2}e^KG_{z\bar z}^{-1}(\partial_z\Omega+K_z\Omega),-i5^{2}(\partial_{\bar z}\bar\Omega+K_{\bar z}\bar\Omega)\rangle = -e^KG_{z\bar z}^{-1}\langle\partial_z\Omega +K_z\Omega,\partial_{\bar z}\bar\Omega +K_{\bar z}{\bar\Omega}\rangle =
\end{equation}

\begin{equation}\label{eq:sigma-pair2b}
=-e^KG_{z\bar z}^{-1}\left(\partial_z\partial_{\bar z}\langle\Omega,\bar\Omega\rangle+K_z\partial_{\bar z}\langle\Omega,\bar\Omega\rangle + K_{\bar z}\partial_z\langle\Omega,\bar\Omega\rangle +K_zK_{\bar z}\langle\Omega,\bar\Omega\rangle\right)=
\end{equation}

\begin{equation}\label{eq:sigma-pair2c}
=-e^KG_{z\bar z}^{-1}\left(-G_{z\bar z}e^{-K} + K_zK_{\bar z}e^{-K}-K_{z}K_{\bar z}e^{-K} - K_zK_{\bar z}e^{-K}+ K_zK_{\bar z}e^{-K} \right)=1
\end{equation}

\begin{equation}\label{eq:sigma-pair3}
\begin{aligned}
\langle -i5^{-3}e^{K}\Omega,-i5^{2}(\partial_{\bar z}\bar\Omega+K_{\bar z}\bar\Omega)\rangle&= -\tfrac{1}{5}e^K\left(\partial_{\bar z}\langle\Omega,\bar\Omega\rangle+K_{\bar z}\langle\Omega,\bar\Omega\rangle\right)\\
&=-\tfrac{1}{5}e^K\left(-K_{\bar z}e^{-K} + K_{\bar z}e^{-K}\right) = 0
\end{aligned}
\end{equation}

\begin{equation}\label{eq:sigma-pair4}
\begin{aligned}
\langle -i5^{-2}e^{K}G_{z\bar z}^{-1}(\partial_z\Omega+K_z\Omega),i5^{3}\bar\Omega\rangle&= 5e^KG_{z\bar z}^{-1}\left(\partial_z\langle\Omega,\bar\Omega\rangle+K_{z}\langle\Omega,\bar\Omega\rangle\right)\\
&=5e^KG_{z\bar z}^{-1}\left(-K_{z}e^{-K} + K_ze^{-K}\right) = 0
\end{aligned}
\end{equation}

The map $\sigma$ allows us to create a relationship between $\mathcal M$ and $\rT$ and create a dictionary assigning regular functions in $\rT$ to the propagators, $K_z$ and $G_{z\bar z}$. This is possible by considering functions $\beta\in \mathcal O_{\rT}$ such that $\beta\circ\sigma$ coincides with the specified quantity. Notice that the section $\sigma$ gives a trivialization of the bundle $\mathcal L$ and $\mathcal L\otimes T\mathcal M$ in the first two coordinates. This allow to write the propagators and terminators as functions of $z$.

We first relate the derivative $\partial_z$ with the vector field $\sf R_0$.

\begin{lemma}\label{lem:Gzz and R0}We have the equalities:
    \begin{equation}
        {\rm d}\sigma(\partial_z)= \frac15\hat G_{z\bar z}{\sf R_0}
    \end{equation}
\begin{equation}
    {\sf Y\circ\sigma} = -5^{-3}G_{z\bar z}^{-3}e^{2K}C_{zzz}\qquad {\sf F\circ\sigma} = -i 5^3e^KG_{z\bar z}^{-2}\Delta_{zz} 
\end{equation}
    In particular, we have that $R_0(f)\circ\sigma = 5G_{z\bar z}^{-1}\partial_z(f\circ\sigma)$ for any regular function $f$ in $\rT$.
\end{lemma}
\begin{proof}
    By the uniqueness statement in Proposition~\ref{prop:R0}, it suffices to show that Gauss-Manin connection applied to ${\rm d}\sigma(\partial_z)$ satisfy equation~\eqref{eq:GM-R0} after multiplying by $\frac15\hat G_{z\bar z}$.

    We start by computing $\nabla_{\partial_z}$, which we denote by $\partial_z$ at each of the $\alpha_i(\sigma(z))$ to get the matrix $\nabla_{{\rm d}\sigma(\partial_z)}$. Three computations are very simple ($\alpha_1,\alpha_3$ and $\alpha_4$)
    \[
    \partial_z(-i5^{-3}e^K\Omega) = -i5^{-3}(K_ze^K\Omega + e^K\partial_z\Omega) = \frac15G_{z\bar z} \alpha_2(\sigma)
\]
\[
    \partial_z(-i5^{2}(\partial_{\bar z}\bar\Omega + K_{\bar z}\bar\Omega ))=-i5^{2}(0 + \partial_zK_{\bar z}\bar\Omega + 0) = -i5^{2} G_{z\bar z}\bar \Omega = -\frac15 G_{z\bar z}\alpha_4(\sigma)
    \]
\[
\partial_z(i5^{3}\bar\Omega)=0
\]
Applying it to $\alpha_2$ gives us exactly the covariant derivative (we use $\partial_z G_{z\bar z}G_{z\bar z}^{-1} = \Gamma_{zz}^z$):
\[
\partial_z(-i5^{-2}G_{z\bar z}^{-1}e^K(K_z\Omega + \partial_z\Omega)) = -i5^{-2}e^KG_{z\bar z}^{-1}D_z^2\Omega
\]

This implies that $\nabla_{{\rm d}\sigma(\partial_z)}$ is of the form of Proposition~\ref{prop:R0}, since, by type, only $\alpha_3$ and $\alpha_0$ will appear on $D^2_z\Omega$. To see this, we apply the pairings with $\alpha_1,\alpha_3$ and $\alpha_4$ and observe, using the Leibniz rule
\[
\langle\partial_z\alpha_2,\alpha_i\rangle = \partial_z\langle\alpha_2,\alpha_i\rangle - \langle\alpha_2,\partial_z\alpha_i\rangle,
\]
that all these are zero except for $i=2$.

For the second part of the theorem, it suffices to actually compute $\langle\partial_z\alpha_2,\alpha_2\rangle$, which is going to give us the coefficient of $\alpha_3$ up to a sign. This implies:
\begin{align*}
\langle\partial_z\alpha_2,\alpha_2\rangle &= \langle\partial_z\alpha_2,-i5^{-2}e^KG_{z\bar z}^{-1}(\partial_z\Omega + K_z\Omega)\rangle = \langle\partial_z\alpha_2,-i5^{-2}e^KG_{z\bar z}^{-1}\partial_z\Omega\rangle\\
&= -i5^{-2}e^KG_{z\bar z}^{-1}(\partial_z\langle\alpha_2,\partial_z\Omega\rangle - \langle\alpha_2,\partial_z^2\Omega\rangle)\\
&=-5^{-4}e^{2K}G_{z\bar z}^{-2}(\partial_z(\langle\partial_z\Omega,\partial_z\Omega\rangle + K_z\langle\Omega,\partial_z\Omega\rangle) - \langle\partial_z\Omega,\partial_z^2\Omega\rangle -K_z\langle\Omega,\partial_z^2\Omega\rangle)\\
&=-5^{-4}e^{2K}G_{z\bar z}^{-2}(\partial_z(-\langle\Omega,\partial_z^2\Omega\rangle) - \langle\partial_z\Omega,\partial_z^2\Omega\rangle)\\
&=-5^{-4}e^{2K}G_{z\bar z}^{-2}(-\langle\partial_z\Omega,\partial_z^2\Omega\rangle)=-5^{-4}e^{2K}G_{z\bar z}^{-2}\langle\Omega,\partial^3_z\Omega\rangle=5^{-4}e^{2K}G_{z\bar z}^{-2}C_{zzz}
\end{align*}

Using that ${\sf R_0}=5G_{z\bar z}^{-1}{\rm d}\sigma(\partial_z)$, we conclude that ${\sf Y\circ\sigma} = -5^{-3}e^{2K}G_{z\bar z}^{-3}C_{zzz}$.

To get the coefficient of $\alpha_0$, we just use the definition on $\alpha_3$:
\[
A\alpha_0 = -i5^{-2}e^KG_{z\bar z}^{-1}D^2_z\Omega +5^{-4}e^{2K}G_{z\bar z}^{-2}C_{zzz}\left(-i5^{2}\bar D_{\bar z}\bar\Omega\right)
\]
By integrating with respect to $\delta_0$, using that $\int_{\delta_0}\Omega=i5^{4}\tau$ and comparing with~\eqref{eq:deltazz-def}, we get:
\[
A = -i\cdot 5^2e^KG_{z\bar z}^{-1}\Delta_{zz}
\]
Using ${\sf R_0} = 5G_{z\bar z}^{-1}{\rm d}\sigma(\partial_z)$ again, we conclude ${\sf F}\circ\sigma =- i\cdot 5^3e^K G_{z\bar z}^{-2}\Delta_{zz} $
\end{proof}
\begin{lemma}[Dictionary]\label{lem:dict}
    Each generator $B$ of the polynomial ring in~\cite{alim_polynomial_2007} can be associated to a regular function on $\hat B$ on $\rT$ such that $\hat B\circ\sigma=B$. This construction allows us to prove equivalence between the anomaly equation of Proposition~\ref{al-HAE} and the equation of Theorem~\ref{alg-HAE-GMCD-Int} that we are going to construct in Theorem~\ref{alg-HAE-GMCD}. See Table~\ref{tab:dictionary}.
\end{lemma}

\begin{proof}
    We first notice that $(5s_0^2)\circ\sigma = e^{-K}$. This comes from the definition of the moduli space $\rT$, where $s_0^2=t_0$ is the choice of normalization of the bundle $\mathcal L$. See~\cite{espreafico_gauss-manin_2026} and~\cite{movasati_modular-type_2015}.

    We can also compute $\hat G_{z\bar z}$ from Lemma~\ref{lem:Gzz and R0}. Indeed, since $\frac{s_4^{10}}{5^5s_0^{10}}\circ\sigma=z$, we have that:
    \[
    {\sf R_0}\left(\frac{s_4^{10}}{5^5 s_0^{10}}\right)\circ\sigma = 5 G_{z\bar z}^{-1}\partial_z(z)\implies\hat G_{z\bar z}^{-1} = \frac15\frac{s_4^{10}(s_4^{10}-s_0^{10})}{s_0^{12}s_5}.
    \]

    This implies $\hat G_{z\bar z} = \dfrac{5s_0^{12}s_5}{s_4^{10}(s_4^{10}-s_0^{10})}$.

    The next step is to compute $\hat K_z$, which can be done by applying $\sf R_0$ to $5s_0^2$. Indeed, we have $K_z = \frac{\partial_z(e^K)}{e^K} = -\frac{\frac15G_{z\bar z}{\sf R_0}(5s_0^2)}{5s_0^2}\circ\sigma$
    \[
    -\frac15 \hat G_{z\bar z}\frac{{\sf R_0}(5s_0^2)}{5s_0^2} = -\frac25\hat{G}_{z\bar z}\frac{{\sf R_0}(s_0)}{s_0}=\dfrac{s_0^{10}\bigl(5^{4}(6s_0^{10}-s_4^{10})+s_0^{2}s_3\bigr)}{s_4^{10}\,\dis}
    \]
    
    Finally, we explain how to compute the propagators and terminators. The main tool is the special geometry relation~\eqref{eq:special-geometry}, that relates the Christoffel symbols and $\partial_{\bar z}S^{zz}$. Indeed, rewrite~\eqref{eq:special-geometry} as:
    \begin{equation}\label{eq:special-geometry-szz}
        \partial_{\bar z}\Gamma^z_{zz} = 2G_{z\bar z} - C_{zzz}\partial_{\bar z}S^{zz}
    \end{equation}
    
    Using $G_{z\bar z} = \partial_{\bar z}K_z$, we conclude:
 \begin{equation}\label{eq:special-geometry-integrated}
      \Gamma_{zz}^z = 2K_z - C_{zzz}S^{zz} + {\rm amb.}
 \end{equation}
    where amb is the holomorphic ambiguity.

    Following~\cite{alim_polynomial_2007}, we choose the ambiguity to be $-1/z$. Since $\Gamma^{z}_{zz} = G_{z\bar z}^{-1}\partial_zG_{z\bar z}$, $\hat\Gamma^z_{zz}$ can be computed via Lemma~\ref{lem:Gzz and R0}. As $C_{zzz}$ has been already computed in terms of $\sf Y$, we can easily compute $S^{zz}$ from~\eqref{eq:special-geometry-integrated}. The same integration strategy can be used for the other propagators $S$ and $S^z$ using the expressions~\eqref{eq:propagators}. $S^z$ can be computed in terms of $\partial_{z}S^{zz}$ and the other generators. $S$ needs $\partial_z{S^z}$.  We, again, choose the ambiguities in the same manner as Alim-Länge. For the terminators, we use equation~\eqref{eq:Deltazz-difeq},~\eqref{eq:Cbar-Deltabar}and~\eqref{eq:terminators} to obtain $\Delta_{zz}=-\Delta^zC_{zzz}$ by choosing the ambiguity to be zero. After that, we can compute $\Delta$ in terms of $\partial_z\Delta^z$. 

    The final results are shown in Table~\ref{tab:dictionary}.
\end{proof}

\begin{lemma}\label{lem:dict-jac}
    The nine regular functions on $\rT$ appearing in Table~\ref{tab:dictionary} form a local étale coordinate system for $\rT$.
\end{lemma}
\begin{proof}
    It suffices to compute the jacobian determinant. Indeed, it is given by $c\frac{s_4^{29}}{s_0^{49}s_5^2}$, which is never zero in $\rT$. In particular, that allows us to define vector fields $\frac{\partial}{\partial \hat S^{zz}}$ on $\rT$ by inverting the matrix.
\end{proof}

\begin{table}[ht]
\centering\small
\renewcommand{\arraystretch}{2}
\resizebox{0.95\textwidth}{!}{
\begin{tabular}{@{}lcc@{}}
\toprule
special geometry on $\mathcal M$ & $\beta\in\mathcal O_{\rT}$ with $\beta\circ\sigma=$ left column  \\
\midrule
$e^{-K}$ & $5\,s_0^{2}$  \\
$z$ & $z=\frac{s_4^{10}}{5^{5}s_0^{10}}$  \\
$G_{z\bar z}$ & $\hat G_{z\bar z}=-\frac{5\,s_0^{12}s_5}{s_4^{10}\,\dis}$  \\
$K_z$ & $\hK_z=\frac{s_0^{10}\bigl(5^{4}(6s_0^{10}-s_4^{10})+s_0^{2}s_3\bigr)}{s_4^{10}\,\dis}$  \\
$S^{zz}$ & $\hS^{zz}=\frac{s_4^{20}\,\dis\,(3s_5-s_0^{2}s_6)}{5^{12}\,s_0^{30}s_5}$ \\
$S^{z}$ & $\hS^{z}=\frac{s_4^{10}\bigl(5^{4}(19s_0^{10}+21s_4^{10})+12s_0^{2}s_3-4s_0^{4}s_2\bigr)}{4\cdot 5^{12}\,s_0^{20}}$ \\
$S$ & $\hS=\frac{
 5^{9}\bigl(13s_0^{20}+21s_0^{10}s_4^{10}+6s_4^{20}\bigr)+4s_0^{4}s_3\bigl(3s_3-s_0^{2}s_2\bigr)+5^{4}\bigl(2s_0^{2}s_3(29s_0^{10}+21s_4^{10})-20s_0^{14}s_2-4s_0^{6}s_1\dis\bigr)}{8\cdot 5^{12}\,s_0^{10}\,\dis}$  \\
$\Delta^{z}$ & $\hD^{z}=\frac{i\,s_4^{10}\,\dis\,s_7}{5^{8}\,s_0^{14}s_5}$  \\
$\Delta$ & $\hD=\frac{i\,s_8}{5^{8}\,s_0^{2}}$  \\
\bottomrule
\end{tabular}}
\caption{The generators of the polynomial ring from~\cite{alim_polynomial_2007} in terms of the $s_i$.}\label{tab:dictionary}
\end{table}
\subsection{Period Map and local expansions}\label{sec: periods}
For each $z$, choose a symplectic basis $\delta$ for the relative homology $H_3(Y_z,C_{\pm},\mathbb Z)$. The period matrix can be computed as

\begin{equation}\label{eq:period-matrix}
{\sf P(s)}=\left[\int_{\delta_i}\alpha_j\right]_{i,j}
\end{equation}

for each point  ${\sf s}=(z,\alpha)\in\rT$. 

Since both $\delta$ and $\alpha$ have constant intersection product, we can obtain relations that determine which matrices can be a period matrix. The group $\rG$ acts on $\rT$ by changing the basis $\alpha$. It also acts on the space of $5\times5$ matrices by right multiplication. In particular, by linearity of integrals, these two actions are compatible. The following result is fundamental for our computations and was proved in~\cite{espreafico_gauss-manin_2026}.

\begin{proposition}%
\label{tau}
For any period matrix $\sf P$ satisfying the relations given by the intersection matrices, there exists a unique $h\in \rG$ such that ${\sf P}\cdot h$ can be written in the form 
\begin{equation}
\label{taumatrix}
\tau = \begin{pmatrix}
1 & \tau_4 & \tau_5 & 0 & 0 \\
0 & \tau_0 & 1 & 0 & 0 \\
0 & 1 & 0 & 0 & 0\\
0& \tau_1 & \tau_3 & 1 & 0\\
0& \tau_2 & -\tau_0\tau_3 +\tau_1& -\tau_0 & 1
\end{pmatrix}.
\end{equation}
for some $\tau_i$.
\end{proposition}

This result means that, for a fixed basis of the homology $\delta$ each orbit of $\rG$ of $\rT$ has exactly one $\alpha$ for which the period matrix is of the form~\eqref{taumatrix}. Locally, the orbit space is just one dimensional and can be locally parametrized by the coordinate $s_4/s_0$, which means it is a 10-fold cover of $\mathcal M$. Let $\sf L_{\delta}$ be the set of all elements $s$ of $\rT$ with matrix of the form $\tau$. $\sf L_{\delta}$ depends on the choice of $\delta$ and, when we go around singular points, $\sf L_{\delta}$ changes according to the monodromy. That gives the $s_i$ certain functional equations with respect to the action of the monodromy group, which is a subgroup of $\mathbb Z^4\rtimes{\rm Sp}(4,\mathbb Z)$. In other words, this means we have a type of quasimodularity satisfied by $s_i$ with respect to the monodromy group (see~\cite{movasati_modular-type_2015} for a discussion in the closed case).

By choosing the basis $\delta$, we are fixing the periods at the initial point and using analytic continuation to obtain expressions for them as $z$ varies. The choice of $\delta$ is therefore the same as the choice of boundary conditions for solving ${\sf R}_0$. Notice however, that different choices of $\delta$ give different local expansions that might not be related by monodromy. 

\begin{proposition}\label{prop:tau-relations}
    Let $\delta=(\delta_0,\dots,\delta_4)$ be basis of $H_3(Y_\psi,C_\pm;\mathbb Z)$ for which the closed part is symplectic. Let ${\sf L}_{\delta}$ be the set of points of $\rT$ whose period matrix is of the form~\eqref{taumatrix}. Locally, in a simply connected neighborhood, we have
    \begin{enumerate}
        \item On ${\sf L}_{\delta}$,
    \begin{align*}
    \tau_0 = \frac{\int_{\delta_1} \alpha_1}{\int_{\delta_2} \alpha_1}, \quad
    \tau_1 = \frac{\int_{\delta_3} \alpha_1}{\int_{\delta_2} \alpha_1}, \quad
    \tau_2 = \frac{\int_{\delta_4} \alpha_1}{\int_{\delta_2} \alpha_1},
    \end{align*}
        \item $\tau_0$ is a local coordinate on ${\sf L}_{\delta}$, ${\sf R}_0|_{{\sf L}_{\delta}}=\frac{\partial}{\partial\tau_0}$, and
        \begin{equation}\label{eq:dz-dtau0}
        \frac{dz}{d\tau_0}={\sf R}_0(z)=\frac{s_4^{10}\left(s_4^{10}-s_0^{10}\right)}{s_0^{12}\,s_5},
        \end{equation}
        which does not vanish on the local neighborhood. In particular, $z$ is also a local coordinate on ${\sf L}_{\delta}$.
        \item As functions of $\tau_0$ on ${\sf L}_{\delta}$, the entries of~\eqref{taumatrix} satisfy
        \begin{equation}\label{eq:tau-relations}
        \frac{d\tau_1}{d\tau_0}=\tau_3,\qquad \frac{d\tau_2}{d\tau_0}=\tau_1-\tau_0\tau_3,\qquad \frac{d\tau_3}{d\tau_0}={\sf Y},\qquad \frac{d\tau_4}{d\tau_0}=\tau_5,\qquad \frac{d\tau_5}{d\tau_0}={\sf F}.
        \end{equation}
    
    \end{enumerate}
\end{proposition}
\begin{proof}
    Most of this proposition is proved in~\cite[Proposition 7]{espreafico_gauss-manin_2026} and~\cite[Proposition 6]{espreafico_gauss-manin_2026}. Statement 1 is equation (33) in the proof of~\cite[Proposition 6]{espreafico_gauss-manin_2026} and Statements 2 and 3 are direct consequences of~\cite[Proposition 7]{espreafico_gauss-manin_2026} (note the misprint that appears there on the expression for $\frac{{\rm d}\tau_2}{{\rm d}\tau_0}$). For the sake of completeness, we reproduce the main ideas of the proof here, without carrying out all the computations. For those, we refer the reader to the reference cited above.

     Statement 1 is just a consequence of Proposition~\ref{tau} and the action of the group $\rG$. By writing the action explicitly, we conclude, for instance, that $\tau_0$ has to be the quotient of positions $(2,3)$ and $(2,2)$ of the period matrix. 

     Statements 2 and 3 use the uniqueness of Theorem~\ref{prop:R0} and the fact that the Gauss-Manin connection satisfies
      \begin{equation}\label{eq:gm-integral}
        {\rm d}\left(\int\omega\right) = \int\nabla\omega.
    \end{equation}
       
    First, notice that the We compute the Gauss-Manin connection on $\sf L_{\delta}$ by using the period matrix. Let $\sf A$ be the Gauss-Manin connection matrix satisfying $\nabla\alpha={\sf A}\alpha$, where $\alpha = (\alpha_0,\,\alpha_1,\,\dots,\,\alpha_4)^{\rm tr}$, i.e., the matrix that expresses $\nabla\alpha_i$ in terms of the other elements of the basis. Using~\eqref{eq:gm-integral}, we can write:
     \[
        {\rm d}\tau = \tau\cdot{\sf A}^{\rm tr}
     \]
     Which implies ${\sf A} = {\rm d}\tau^{\rm tr}\cdot (\tau^{-1})^{\rm tr}$. Applying the connection to the vector field $\frac{\rm d}{{\rm d}\tau}$ and using Griffiths transversality, we get:
       \[
        \nabla_{\frac{\rm d}{{\rm d}\tau}} =
\begin{pmatrix}
0 & 0 & 0 & 0 & 0 \\
0 & 0 & 1 & 0 & 0 \\
\frac{d\tau_5}{d\tau_0} & 0 & 0 & \frac{d\tau_3}{d\tau_0} & 0 \\
0 & 0 & 0 & 0 & -1 \\
0 & 0 & 0 & 0 & 0
\end{pmatrix}.
    \]

    A direct application of~\ref{prop:R0} gives ${\sf R_0}|_{\sf L_{\delta}}=\frac{\rm d}{{\rm d}\tau}$. Griffiths transversality gives us the relations 1, 2 and 4 in Statement 3. Comparing with ${\sf R_0}$ gives us the other two relations. Equation~\eqref{eq:dz-dtau0} is just applying $\sf R_0$ to $\frac{s_4^10}{5^5s_0^10}$ as in the proof of~\ref{lem:dict}.
\end{proof}
\begin{corollary}\label{cor:tau0-mirror-map}
    The functions $s_i|_{{\sf L}_{\delta}}$ are the solutions of~\eqref{eq:R0} with $\frac{d}{d\tau_0}$ in place of the dot, and their initial conditions are determined by $\delta$. 
    
    Since $\tau_0$ is exactly the quotient of the periods, the local expansion in it correspond to the expansions around the flat coordinates at the special points once the initial conditions are chosen. This finishes the proof of~\ref{prop:expansions}: the conditions chosen there are equivalent to fixing a basis of the homology so the solutions of the Picard-Fuchs equation are as in~\ref{pic-fuchs} and~\ref{sec: conif}.
\end{corollary}

\subsection{The AMSY Lie Algebra and Anomaly Equation }

\begin{theorem}\label{thm:VX}
    Let $p=(Y,C_{\pm},\alpha)$ and let $X\in\mathfrak h$. 
    We can define the vector field $V_X$ at $p$ as $V_X = \frac{\rm d}{{\rm d}u}\left(\exp(uX)\cdot p\right)$.

    The operator $\nabla_{V_X}$ at $p$, represented in the basis $\alpha$, has the same matrix as the transpose of $X$. Besides that, $V_X$ is the unique vector field $V$ satisfying
   \begin{equation}\label{vector-lie}
       \nabla_{V} = X^{\rm tr}
   \end{equation}
\end{theorem}
\begin{proof}
    One way to see this is to compute directly using the action of the group $\rG$ in the coordinates $s_i$ and computing the vector fields $V_X$ directly and then the Gauss-Manin connection $\nabla_{V_X}$ in the same coordinates and check that the matrices are the same for any element of the basis of $\mathfrak h$. Using that $\frac{d}{du}\exp(uX)=\exp(uX)\cdot X$, we can use linearity arguments to extend it to the whole Lie algebra. 

    However, the result can be proved directly, without computing explicitly, using the period language from~\ref{sec: periods}. Since, along the integral curve $\gamma: u\mapsto\exp(uX)\cdot p$, the mirror quintic is unchanged, we can fix a basis of the homology and compute the period matrix. By definition of the action, we have
    \[
        {\sf P}(\exp(uX)\cdot p) = {\sf P}(p)\cdot \exp(uX)
    \]
     Using equation~\eqref{eq:gm-integral}, we have and computing ${\rm d}$, we have:
    \[
    {\sf P}(\nabla(\exp(uX)\cdot p)) = {\rm d}({\sf P}(\exp(uX)\cdot p)) = {\sf P}(\exp(uX)\cdot p)\cdot X{\rm d}u
    \]
    Since $V_X$ is $\frac{{\rm d}}{{\rm d}u}$ along $\gamma$ we have, after applying $u=0$:
    \[
    {\rm P}(\nabla_{V_X}(p)) = {\rm P}(p)\cdot X
    \]

    Since the basis of the homology is fixed, this implies that the Gauss-Manin connection acts as $X$. Rewriting it in the notation of~\ref{prop:R0}, we have
    \[
\nabla_{V_X} = X^{\rm tr}
    \]

    For the uniqueness statement, it suffices to see that the period map is locally injective (local coordinate in Proposition~\ref{tau}) and that any $V$ that satisfy~\eqref{vector-lie} is tangent to the curve $\gamma:u\mapsto\exp(uX)\cdot p$ for any $p$.
\[
    {\sf P}(\exp(uX)\cdot p)\cdot X = {\sf P}(\nabla_{V}(\exp(uX)\cdot p)) = {\sf P}(\exp(uX)\cdot p)\cdot X{\rm d}u(V|_{\gamma})
    \]
    Restricting to $\gamma$, we have $V|_{\gamma}=\frac{\rm d}{{\rm d}u}$, which implies $V (p) = \frac{\rm d}{{\rm d}u}\gamma(u)|_{u=0} = V_X(p)$.
\end{proof}
For the closed case (involving absolute cohomology), the existence and uniqueness statement was proved in~\cite{alim_gaussmanin_2016} and extended to general Dwork families in~\cite{nikdelan_modular_2020}. We just add the simple observation that the $\sf R_i$ can be expressed as exponentials.

Consider the following basis for the lie algebra:
{\small
\begin{align*}
X_1 &= 
\begin{bmatrix}
0 & 0 & 0 & 0 & 0 \\
0 & 0 & 0 & 0 & 0 \\
0 & 0 & -1 & 0 & 0 \\
0 & 0 & 0 & 1 & 0 \\
0 & 0 & 0 & 0 & 0
\end{bmatrix},
&
X_2 &=
\begin{bmatrix}
0 & 0 & 0 & 0 & 0 \\
0 & -1 & 0 & 0 & 0 \\
0 & 0 & 0 & 0 & 0 \\
0 & 0 & 0 & 0 & 0 \\
0 & 0 & 0 & 0 & 1
\end{bmatrix},
&
X_3 &=
\begin{bmatrix}
0 & 0 & 0 & 0 & 0 \\
0 & 0 & -1 & 0 & 0 \\
0 & 0 & 0 & 0 & 0 \\
0 & 0 & 0 & 0 & 1 \\
0 & 0 & 0 & 0 & 0
\end{bmatrix},
\\[6pt]
X_4 &=
\begin{bmatrix}
0 & 0 & 0 & 0 & 0 \\
0 & 0 & 0 & 0 & 0 \\
0 & 0 & 0 & 1 & 0 \\
0 & 0 & 0 & 0 & 0 \\
0 & 0 & 0 & 0 & 0
\end{bmatrix},
&
X_5 &=
\begin{bmatrix}
0 & 0 & 0 & 0 & 0 \\
0 & 0 & 0 & 1 & 0 \\
0 & 0 & 0 & 0 & 1 \\
0 & 0 & 0 & 0 & 0 \\
0 & 0 & 0 & 0 & 0
\end{bmatrix},
&
X_6 &=
\begin{bmatrix}
0 & 0 & 0 & 0 & 0 \\
0 & 0 & 0 & 0 & 1 \\
0 & 0 & 0 & 0 & 0 \\
0 & 0 & 0 & 0 & 0 \\
0 & 0 & 0 & 0 & 0
\end{bmatrix},
\\[6pt]
X_7 &=
\begin{bmatrix}
0 & 0 & 0 & 1 & 0 \\
0 & 0 & 0 & 0 & 0 \\
0 & 0 & 0 & 0 & 0 \\
0 & 0 & 0 & 0 & 0 \\
0 & 0 & 0 & 0 & 0
\end{bmatrix},
&
X_8 &=
\begin{bmatrix}
0 & 0 & 0 & 0 & 1 \\
0 & 0 & 0 & 0 & 0 \\
0 & 0 & 0 & 0 & 0 \\
0 & 0 & 0 & 0 & 0 \\
0 & 0 & 0 & 0 & 0
\end{bmatrix}.
\end{align*}}

We define the vector fields ${\sf R}_i$ such that $\nabla_{{\sf R}_i}=X_i^{\rm tr}$.

\begin{theorem}\label{thm:lie-algebra}
    The vector fields ${\sf R}_i$ and ${\sf R}_0$ (the modular vector field) generate a Lie algebra (called AMSY Lie algebra) over $\mathcal O_{\rT}$ with respect to the commutator. The subalgebra generated by the ${\sf R}_i$, $i=1,\dots 8$ is isomorphic to $\mathfrak h$. We present the multiplication table as Table~\ref{mult-table-lie}. 
   \begin{table}[H]
     {\footnotesize
\begin{equation*}\begin{array}{|c|c|c|c|c|c|c|c|c|c|}
\hline
    & {\sf R}_0 & {\sf R}_1 & {\sf R}_2 & {\sf R}_3 & {\sf R}_4 & {\sf R}_5 & {\sf R}_6 & {\sf R}_7 & {\sf R}_8 \\ \hline
{\sf R}_0 & 0   & {\sf R}_0         & -{\sf R}_0        & -{\sf R}_2 + {\sf R}_1      & Y {\sf R}_1+F{\sf R}_7      & 2{\sf R}_4 + Y{\sf R}_3+F{\sf R}_8 & {\sf R}_5        & 0   & {\sf R}_7 \\ \hline
{\sf R}_1 & -{\sf R}_0 & 0          & 0           & {\sf R}_3             & -2{\sf R}_4           & -{\sf R}_5               & 0          & -{\sf R}_7  & 0 \\ \hline
{\sf R}_2 & {\sf R}_0 & 0           & 0           & -{\sf R}_3            & 0               & -{\sf R}_5               & -2{\sf R}_6      & 0   & -{\sf R}_8 \\ \hline
{\sf R}_3 & {\sf R}_2 - {\sf R}_1 & -{\sf R}_3  & {\sf R}_3         & 0               & -{\sf R}_5            & -2{\sf R}_6              & 0          & -{\sf R}_8   & 0 \\ \hline
{\sf R}_4 & -Y{\sf R}_1-F{\sf R}_7 & 2{\sf R}_4 & 0        & {\sf R}_5             & 0               & 0                  & 0          & 0   & 0 \\ \hline
{\sf R}_5 & -2{\sf R}_4 - Y {\sf R}_3-F{\sf R}_8 & {\sf R}_5 & {\sf R}_5 & 2{\sf R}_6           & 0               & 0                  & 0          & 0   & 0 \\ \hline
{\sf R}_6 & -{\sf R}_5 & 0         & 2{\sf R}_6        & 0               & 0               & 0                  & 0          & 0   & 0 \\ \hline
{\sf R}_7 & 0   &{\sf R}_7          & 0           & {\sf R}_8               & 0               & 0                  & 0          & 0   & 0 \\ \hline
{\sf R}_8 & -{\sf R}_7  & 0          &  {\sf R}_8         & 0               & 0               & 0                  & 0          & 0   & 0 \\ \hline
\end{array}
\end{equation*}}
\caption{Multiplication table for the Lie Algebra generated by the $\sf R_i$}\label{mult-table-lie}
\end{table}
\end{theorem}

The main motivation to introduce such Lie algebra is to give an intrinsic and algebraic framework for the topological string partition functions and the holomorphic anomaly equations. This idea was mainly developed in~\cite{alim_gaussmanin_2016}, and gives a purely algebraic (in terms of the moduli space of Hodge structures) description of the anomaly equations as emanating from the Lie Algebra. We explain this in the following pages.
\begin{proof}
    Computing the vector fields via~Theorem~\ref{thm:VX}, we obtain
{\small
\begin{align*}
{\sf R}_1 &= s_{5}\,\frac{\partial}{\partial s_{5}}
+ s_{6}\,\frac{\partial}{\partial s_{6}}
+ s_{7}\,\frac{\partial}{\partial s_{7}} \\[4pt]
{\sf R}_2 &= \frac{s_{0}}{2}\,\frac{\partial}{\partial s_{0}}
+ 2s_{1}\,\frac{\partial}{\partial s_{1}}
+ 3s_{2}\,\frac{\partial}{\partial s_{2}}
+ 4s_{3}\,\frac{\partial}{\partial s_{3}}
+ \frac{s_{4}}{2}\,\frac{\partial}{\partial s_{4}}
+ 3s_{5}\,\frac{\partial}{\partial s_{5}}
+ 2s_{6}\,\frac{\partial}{\partial s_{6}}
+ s_{8}\,\frac{\partial}{\partial s_{8}} \\[4pt]
{\sf R}_3 &= -\,\frac{s_{5}\left(3125\,s_{0}^{6}
+ s_{2}
+ \dfrac{(-3125\,s_{0}^{8}-s_{3})\,s_{6}}{s_{5}}\right)}
{625\left(-s_{0}^{10} + s_{4}^{10}\right)}\,
\frac{\partial}{\partial s_{1}}
+ s_{6}\,\frac{\partial}{\partial s_{2}}
+ s_{5}\,\frac{\partial}{\partial s_{3}}
+ s_{7}\,\frac{\partial}{\partial s_{8}} \\[4pt]
{\sf R}_4 &= \frac{625\left(s_{0}^{10} - s_{4}^{10}\right)}{s_{5}}\,
\frac{\partial}{\partial s_{6}} \\[4pt]
{\sf R}_5 &= -\,\frac{3125\,s_{0}^{8} + s_{3}}{s_{5}}\,
\frac{\partial}{\partial s_{1}}
- \frac{625\left(-s_{0}^{10} + s_{4}^{10}\right)}{s_{5}}\,
\frac{\partial}{\partial s_{2}} \\[4pt]
{\sf R}_6 &= \frac{\partial}{\partial s_{1}} \\[4pt]
{\sf R}_7 &= \frac{\partial}{\partial s_{7}} \\[4pt]
{\sf R}_8 &= \frac{\partial}{\partial s_{8}}
\end{align*}
}
The result follows by direct computation.
\end{proof}

\begin{remark}\label{rem:AMSY}
    The first six vector fields ${\sf R}_i$ correspond to the ones in the last section of~\cite{alim_gaussmanin_2016}, when considering the closed case. The last two appear in the real case. 
\end{remark}

The next step is to show that this Lie Algebra corresponds to Walcher's extended holomorphic anomaly equation, i.e., that the generating functions $\mathcal G^{\chi}$ satisfy a system of equations involving the ${\sf R}_i$. The Lie algebra structure allows us to easily obtain all the derivatives with respect to the $s_i$ and then solve it by partial integration.

\begin{lemma}\label{der-prop}
    We define $G^{\chi}=i^{\chi}s_0^{2\chi}\hat{{\mathcal G}}^{\chi}$, where $\hat{\mathcal{G}^\chi}$ is the regular function on $\rT$ obtained after applying the dictionary~\ref{tab:dictionary}.
    
    We can relate the vector fields $\sf R_i$ with the derivations on $\mathcal M$. 
    \begin{align}
       5i^{\chi}G_{z\bar z}^{-1} D_z \mathcal G^\chi = (s_0^{-2\chi}{\sf R}_0G^\chi)\circ\sigma\\
        5^{-6}G_{z\bar z}^{-2}\frac{\partial}{\partial S^{zz}}= (-s_0^{4}{\sf R}_4(-))\circ\sigma\\
       5^{-7}G_{z\bar z}^{-1} \frac{\partial}{\partial \Delta^z} = (i\ s_0^{2}{\sf R}_7(-))\circ\sigma
    \end{align}

    In particular, these expressions lift to expressions involving the regular functions on the left of Table~\ref{tab:dictionary} and the derivatives defined in Lemma~\ref{lem:dict-jac}.
\end{lemma}
    \begin{proof}
        We compute directly:
        
\begin{equation}\label{eq:derprop-1}
\begin{aligned}
({\sf R}_0G^{\chi})\circ\sigma&=i^{\chi}\left(\big({\sf R}_0(s_0^{2\chi})\big)\circ\sigma\cdot\mathcal G^{\chi}+(s_0\circ\sigma)^{2\chi}\cdot\big({\sf R}_0(i^{-\chi}s_0^{-2\chi}G^{\chi})\big)\circ\sigma\right)\\
&=i^{\chi}\left(\big(2\chi s_0^{2\chi-1}{\sf R}_0(s_0)\big)\circ\sigma\cdot\mathcal G^{\chi}+(s_0\circ\sigma)^{2\chi}\cdot\big({\sf R}_0(i^{-\chi}s_0^{-2\chi}G^{\chi})\big)\circ\sigma\right)
\end{aligned}
\end{equation}

        By the discussion above, Lemma~\ref{lem:Gzz and R0} and the explicit formula of ${\sf R}_0$ in Equation~\eqref{eq:R0}, we have
        
\begin{equation}\label{eq:derprop-2}
({\sf R}_0G^{\chi})\circ\sigma=i^{\chi}\left(\chi\left(s_0^{2\chi-2}\,\frac{3750s_0^{10}+s_0^2s_3-625s_4^{10}}{s_5}\right)\circ\sigma\cdot\mathcal G^{\chi}+(s_0\circ\sigma)^{2\chi}5G_{z\bar z}^{-1}\cdot\partial_z\mathcal G^{\chi}\right)
\end{equation}

        Using that $e^{-K}$, $G_{z\bar z}$ and $K_z$ can be expressed in terms of the variables $s_i$ as in~\ref{tab:dictionary}, we write:
        
\begin{equation}\label{eq:derprop-3}
\begin{aligned}
(s_0^{-2\chi}{\sf R}_0(G^{\chi}))\circ\sigma
&=i^{\chi}\left(-5\chi K_zG_{z\bar z}^{-1}\mathcal G^{\chi}+5G_{z\bar z}^{-1}\partial_z\mathcal G^{\chi}\right)\\
&=5i^{\chi}G_{z\bar z}^{-1}\cdot (-\chi K_z \mathcal G^{\chi}+\partial_z\mathcal G^{\chi})=5i^{\chi}G_{z\bar z}^{-1}D_z\mathcal G^{\chi}
\end{aligned}
\end{equation}

        For the second equality, we use that $S^{zz}$ can be written in terms of the $s_i$ as 
\begin{equation}\label{eq:derprop-Szz}
S^{zz}=\left(\frac{s_4^{20} \left(-s_0^{10} + s_4^{10}\right) \left(-3 s_5 + s_0^2 s_6\right)}{244140625\, s_0^{30} s_5}\right)\circ\sigma=G_{z\bar z}^{-1}\cdot\left( \dfrac{s_4^{10} \left(-3 s_5 + s_0^2 s_6\right)}{48828125\, s_0^{18}}\right)\circ\sigma.
\end{equation}

We use that this is the only variable among the propagators whose corresponding regular function contains $s_6$. That means that, when computing $(\partial_{s_6}(-))\circ\sigma$ in terms of the derivatives with respect to the propagators and $K_z$, we have
        
\begin{equation}\label{eq:derprop-ds6}
\left(\frac{\partial}{\partial s_6}(-)\right)\circ\sigma = \left(\frac{\partial \hat S^{zz}}{\partial s_6}\right)\circ\sigma\cdot\frac{\partial}{\partial S^{zz}}
\end{equation}

    Computing that, we get 
\begin{equation}\label{eq:derprop-dSzz}
G_{z\bar z}^{-1}\frac{\partial}{\partial S^{zz}}=\left(\frac{5^{11} s_0^{16}}{s_4^{10}}\frac{\partial}{\partial s_6}(-)\right)\circ\sigma.
\end{equation}
 Multiplying by $5^{-6}G_{z\bar z}^{-1}$, we obtain
        
\begin{equation}\label{eq:derprop-R4}
5^{-6}G_{z\bar z}^{-2}\frac{\partial}{\partial S^{zz}}=\left(5^{4}s_0^{4}\left(\frac{s_4^{10}-s_0^{10}}{s_5}\right)\frac{\partial}{\partial s_6}(-)\right)\circ\sigma = (-s_0^{4}{\sf R}_4(-))\circ\sigma
\end{equation}

        The same analysis can be made for $\Delta^{z}$. The table~\ref{tab:dictionary} gives
        
\begin{equation}\label{eq:derprop-Deltaz} \Delta^{z}=\left(\frac{i\,s_4^{10} (s_0^{10} - s_4^{10}) s_7}{390625 s_0^{14} s_5}\right)\circ\sigma=G_{z\bar z}^{-1}\cdot\left(-i\,5^{-7}s_0^{-2}s_7\right)\circ\sigma 
\end{equation}

       By using that $\left(\frac{\partial}{\partial s_7}(-)\right)\circ\sigma= \left(\frac{\partial\hat \Delta^{z}}{\partial s_7}\right)\circ\sigma\cdot\frac{\partial}{\partial \Delta^{z}}$ and that ${\sf R}_7 = \frac{\partial}{\partial s_7}$, we conclude:
       
\begin{equation}\label{eq:derprop-R7}
5^{-7}G_{z\bar z}^{-1} \frac{\partial}{\partial \Delta^{z}} = (i\,s_0^{2}{\sf R}_7(-))\circ\sigma
\end{equation}

    \end{proof}

\begin{lemma}\label{lem:R3}
    Using the same assumptions and definitions as the lemma above, we have:

\begin{equation}\label{eq:R3-lemma}
-\frac15 G_{z\bar z}\left(\frac{\partial \mathcal G^{\chi}}{\partial K_z} + S^z \frac{\partial \mathcal G^{\chi}}{\partial S} + S^{zz} \frac{\partial \mathcal G^{\chi}}{\partial S^z} + \Delta^z \frac{\partial \mathcal G^{\chi}}{\partial \Delta}\right) = ({\sf R}_3\mathcal G^{\chi})\circ\sigma = 0
\end{equation}
\begin{proof}
    The strategy is exactly the same. We use the expressions in Table~\ref{tab:dictionary} and Lemma~\ref{lem:dict-jac} to write the expression on the left hand side of~\eqref{eq:R3-lemma} in terms of $s_i$. After that we compare with $\sf R_3$. By Proposition~{\ref{al-HAE}}, we get zero.
\end{proof}
\end{lemma}
\begin{theorem}\label{alg-HAE-GMCD}
    The B-side generating functions $G^{\chi}$, $\chi>0$  of Euler characteristic $\chi$ are regular functions on $\rT$ in the variables $s_0,\dots,s_8$ given as solutions to the recursive differential equations:
    \begin{align}
        {\sf R}_i G^{\chi} = 0, i=1,3\\
        {\sf R}_2 G^{\chi} = \chi G^{\chi}\\
        {\sf R}_4 G^{\chi} = \frac1{5^8} \left(\sum_{\substack{\chi_1+\chi_2 = \chi-2\\ \chi_1,\chi_2\geq0}}{\sf R}_0G^{\chi_1}{\sf R}_0G^{\chi_2}+{\sf R}_0^2G^{\chi-2}\right)\\
        {\sf R}_7 G^{\chi} = -\frac1{5^8}{\sf R}_0 G^{\chi-1}
    \end{align}
with ${\sf R}_0G^0$ and ${\sf R}_0^2G^{-1}$ and given explicitly as 

\begin{equation}\label{eq:seeds}
\begin{gathered}
{\sf R}_0G^{0}=\frac{\dis^{2}s_7^{2}}{4s_5^{3}}+\frac{625\,\dis\,s_6}{2s_5^{2}}-\frac{5^{4}(1022s_0^{10}+3s_4^{10})+200s_0^{2}s_3}{48\,s_0^{2}s_5},\\
{\sf R}_0^{2}G^{-1}=\tfrac12{\sf F}.
\end{gathered}
\end{equation}
\end{theorem}    
\begin{proof}
    The proof is a consequence of Lemma~\ref{der-prop} and Lemma~\ref{lem:R3} and the holomorphic anomaly equation written as in Proposition~\ref{al-HAE}. Notice that the expressions from the lemmas lift to expressions in terms of the regular $\rT$ (by Lemma~\ref{lem:dict-jac}) and therefore we have that the equalities obtained are not only true on the image of $\sigma$. We just need to justify the annihilation of $\sf R_1$ and the equation involving $\sf R_2$. That can be seen from the fact that all the generators $S^{zz}$, $S^{z}$, $S$, $\Delta^z$, $\Delta$ and $K_z$ are annihilated by $\sf R_1$ and $\sf R_2$. This implies that $\hat{\mathcal G^{\chi}}$ is homogeneous of degree zero and since we are multiplying by $s_0^{2\chi}$ to obtain $G^{\chi}$, it's going to be homogeneous of degree $\chi$ and satisfy equation $\chi$. Here degree is meant with the weights appearing in the expression of $\sf R_2$ in Theorem~\ref{thm:lie-algebra}.
\end{proof}

Notice that, although the equation involves only some of the ${\sf R}_i$, the table of the commutators implies that we can compute any vector ${\sf R}_i$ of $G^\chi$, so we have a complete system. Also, since ${\sf R}_1$ to ${\sf R}_8$ come directly from the exponential from the group that has an eight dimensional orbit, there should be a one-dimensional freedom on the solutions of the equation, which is exactly the ambiguity.

\begin{theorem}\label{thm:ambiguity}
    The solutions of the differential equation from the theorem above are determined up to an ambiguity of the form
    
\begin{equation}\label{eq:amb-form}
s_0^{2\chi} h(s_0,s_4)
\end{equation}

   Where $h$ is a regular function on $\rT$ and is homogeneous of degree zero, i.e, it is actually a function in $v=s_4/s_0$.

\end{theorem}
\begin{proof}
    Assume $G^{\chi}$ is a solution. Let $\tilde{G}^{\chi} = G^{\chi}+f$ be another solution. The equation implies that $f$ does not depend of $s_6$ or $s_7$ because ${\sf R}_4f={\sf R}_7f=0$. Now, ${\sf R}_1 f=0$ implies $f$ does not depend on $s_5$, simply using the previous two facts and the expression for ${\sf R}_1$. To finish, we use the brackets $[{\sf R}_3,{\sf R}_4]=-{\sf R}_5$, $[{\sf R}_3,{\sf R}_5]=-2{\sf R}_6$ and $[{\sf R}_3,{\sf R}_7]=-{\sf R}_8$. Since ${\sf R}_3$ and ${\sf R}_4$ vanish on $f$, we get that ${\sf R}_5, {\sf R}_6$ and ${\sf R}_8$ also vanish on $f$. This implies that $f$ does not depend on $s_8$, $s_1$ and $s_2$. This implies ${\sf R}_3(f)=s_5\partial_{s_3}f = 0$ which implies $f$ does not depend on $s_3$ either. We conclude that $f$ is a regular function of $s_0$ and $s_4$.

    The homogeneity part of the argument comes from ${\sf R}_2$. It implies ${\sf R}_2(f)=\chi f$, which shows $f$ has to be degree $\chi$. Dividing by $s_0^{2\chi}$, we get the result.
\end{proof}

\begin{remark}\label{rem:ambiguity}
    Notice that the ambiguity is a one dimensional object. Notice ${\sf R}_i h=0$ for all $i$ except for ${\sf R}_0$, which means that it can be seen as a function on $\sf L$, and therefore, a function of $z$ or $\tau$. It can be fixed by studying its expansions at each point, since the expansions have properties that they are supposed to satisfy. In the next section, we quickly discuss how to fix the ambiguity for the real case.
\end{remark}

\begin{theorem}\label{partial-int}
    The anomaly equation and the Lie Algebra structure allows us to compute the derivatives with respect to the $s_i$ and solve the equation by partial integration.
\end{theorem}
\begin{proof}
    $\partial_{s_6}$ and $\partial_{s_7}$ are computed right away in the equation. ${\sf R}_1G^{\chi}=0$ gives us $\partial_{s_5}$ after using the values for $\partial s_7$ and $\partial s_6$. Using that $[{\sf R}_3,{\sf R}_4]=-{\sf R}_5$ and $[{\sf R}_3,{\sf R}_5]=-2{\sf R}_6$, we can write ${\sf R}_5(G^{\chi}) = -{\sf R}_3{\sf R}_4(G^{\chi})$, which gives ${\sf R}_5(G^\chi)$ and then ${\sf R}_6(G^{\chi}) = -\frac12 {\sf R}_3{\sf R}_5(G^{\chi})$. But ${\sf R}_6$ is simply $\partial_{s_1}$ and this also allows computation of $\partial_{s_2}$ from the expression for ${\sf R}_5$. Finally, we use that $[{\sf R}_7,{\sf R}_3]={\sf R}_8$ to compute $\partial_{s_8}$ and ${\sf R}_3$ itself allows computation of $\partial_{s_3}$.

    This gives us a simple exact differential equation that can be solved directly by integrating in each variable, adding a constant and matching the constant to the next variable etc. 
\end{proof}

The theorem above, besides having a computational application, also gives some insight about the anomaly equation. Geometrically, the variables $s_i$, $i\ne0,4$ measure the choice of a basis of for the third cohomology that respects the Hodge structure. So the equation determines how the Hodge structure varies in terms of the complex structure variation (${\sf R}_0$ that appears on the RHS) of the lower genus amplitudes.

\section{Ambiguity}\label{sec:ambiguity}
In this section, our goal is to explain the algorithm to compute the amplitudes $G^{\chi}$. In order to be able to fix it, we make some assumptions that are standard in the Physics literature. For the closed invariants on the quintic, the works\cite{guo_structure_2018,chang_conifold_2026,chang_bcov_2026,chang_polynomial_2021} studied these properties on the A-model. For real invariants, the author is not aware of any works.

The first assumption we make is that the choice of ambiguity yields solutions $G^{\chi}$ that satisfy~\ref{eq:BPS-real}. 
\begin{lemma}\label{lem:amb-parity}
   For each nonzero monomial $m$ of $\rT$ we define $N(m)$ as the sum of the degrees of $m$ with respect to $s_4$, $s_7$ and $s_8$ modulo 2. If $G^{\chi}$ satisfy~\ref{eq:BPS-real}, then $N(m)\equiv\chi\mod2$ for all monomials $m$ of $G^{\chi}$. In particular, all monomials of the ambiguity $h^{\chi}$ have the degree corresponding to $s_4$ congruent to $\chi$.
\end{lemma}
\begin{proof}
    Fix a monomial $m$. Note that all nonzero monomials in ${\sf R_i}(m)$ have the same value of $N$. Indeed, write ${\sf R_i}=\sum f_i\partial_{s_i}$. Take $i\ne7,8$. First, we consider the derivatives with respect to $s_i$ for $i\ne4,7,8$. For them it suffices to look at the monomials of $f_i$, since the do not affect $N(m)$ (if $N(m)\ne0$). Since all such $f_i's$ have only even powers of $s_4$ (indeed, multiple of ten), the product $f_i\cdot m$ will have all monomials with same $N(m)$. For $i=4,7,8$, we have to make sure that the accompanying $f_i$ has all monomials with $N=1$, which is readily checked by looking at the definitions. For $i=7,8$, if ${\sf R_i}(m)\ne0$, it has only one monomial $u$ and $N(u)=N(m)-1$.

    Now assume $G^{\chi}$ has all monomials with $N(m)=\chi\mod 2$ for $\chi<\chi_0$. The partial derivatives of $G^{\chi_0}$ are computed in terms of the ${\sf R_i}(G^{\chi_0-1})$ and ${\sf R_i}(G^{\chi_0-2})$ from Theorem~\ref{partial-int}. In particular, the expression for ${\sf R_4}G^{\chi_0}$ in terms of $G^{\chi_0-2}$ implies that all monomials with $s_6$ have $N(m)=\chi_0-2=\chi_0\mod2$, since $\sf R_0$. Proceeding similarly for all partial derivatives, we conclude that all monomials except for the ones with $s_4$ and $s_0$ (ambiguity) have $N(m)=\chi_0\mod2$. 

    To finish, we use that $G^{\chi_0}$ satisfies~\ref{eq:BPS-real}. This implies that the $q$-expansion has to have exponents $q^{d/2}$ with $d=\chi_0\mod2$. Since only $s_7,s_8$ and $s_4$ have $q$-expansions involving $q^{1/2}$, we conclude that the terms $q^{d/2}$ appearing in the expansion of a monomial have $d=N(m)\mod2$. Since all monomials except for the ambiguity satisfy $N(m)=\chi_0\mod2$, we conclude that the monomials in the ambiguity also have to satisfy $N(m)=\chi_0\mod2$.
\end{proof}
\begin{remark}
    The last step of the proof also implies that the minimum power of $s_4$ is five, since we need expansions in $q^{1/2}$.
\end{remark}

To obtain an ansatz for the ambiguity, we need a few hypotheses about the growth of it around the special points. The BPS expansion~\ref{eq:BPS-real} implies regularity at $z=0$. At the conifold point, we expect a pole of order $z^\chi$ and, at $z=\infty$, a pole of order at most $\frac\chi2$. 

\subsection{Assumptions on the Ambiguity}\label{sec:amb-assump}
We start by finding an ansatz for the ambiguity. We use the expected growth in the variable $z$ on the orbifold and on the conifold to be, respectively $\chi/2$ and $\chi$. This is an adaptation from the ansatz from~\cite{yamaguchi_topological_2004} used by~\cite{huang_topological_2009} to fix the ambiguities in the complex case. These conditions were also used in~\cite{alim_polynomial_2007} and~\cite{konishi_solutions_2007} as ansatz for the ambiguity. 

\begin{proposition}\label{amb-ansatz} Assuming that the order of the poles at the orbifold and conifold are at most $\chi/2$ and $\chi$ respectively, and that it is regular at large complex structure, the ambiguity $h^{\chi}$ can be written in the following form
    
\begin{equation}\label{eq:amb-ansatz-final}
h^{\chi}(s_0,s_4) \;=\; \frac{1}{s_0^{5\chi}}\!\!\sum_{\substack{0\le k\le 3\chi\\ k\equiv\chi\ (2)}}\!\!
  b_k\, s_4^{5k}\,\bigl(s_0^{10}-s_4^{10}\bigr)^{\frac{\chi-k}{2}} .
\end{equation}

\end{proposition}
\begin{proof}
   By Lemma~\ref{lem:amb-parity}, we can assume that $h^{\chi}$ is a function of $v=(\frac{s_4}{s_0})^5$ and that only exponents with the same parity as $\chi$ appear in the expression. Using the growth condition at $z=\infty$ and $z=5^{-5}$, we can write, for $\chi$ even:
    \[
       \sum_{j=0}^{\chi/2}b_jv^{2j}+ \sum_{\ell=1}^{\chi}\frac{c_{\ell}}{(1-v^2)^{\ell}} =\sum_{k=0,\,k\text{ even}}^{\chi}b_kv^k+ \sum_{k=\chi+2}^{3\chi}c_{k}(1-v^2)^{\frac{\chi-k}2},
    \]
    afte writing $2\ell = k-\chi$ and $2j=k$. We can make both sums into one to get~\eqref{eq:amb-ansatz-final}. 
    
    For odd $\chi$, we do not have a term $v^0$. We can do the same but go until $(\chi+1)/2$. 
   \[
       \sum_{j=0}^{(\chi-1)/2}b_jv^{2j+1}+ \sum_{\ell=1}^{\chi}\frac{c_{\ell}}{(1-v^2)^{\ell}}=
       \sum_{k=1,\,k\text{ odd}}^{\chi}b_kv^{k}+ \sum_{k=\chi+2}^{3\chi}c_{k}(1-v^2)^{\frac{\chi-k}2}
    \]
    After writing $2\ell = k-\chi$ and $2j+1=k$. Again, we can put it together to obtain expression~\eqref{eq:amb-ansatz-final}.
\end{proof}

With the ansatz~\eqref{eq:amb-ansatz-final}, we can refine the conditions we have for the conifold and use some geometric information from the A-model to fix the coefficients $b_k$ for each $\chi$. There are three main properties the ambiguity $h^\chi(s_0,s_4)$ has to force $G^{\chi}$ to satisfy. Note that, since the properties are stated for $\mathcal G_A^{\chi}$ rather than $G^\chi$, we have to renormalize by multiplying by $(-5i)^\chi$ (the factor 5 relating $s_0$ and $e^{-K}$ and the complex unit that has to disappear).

\begin{enumerate}
    \item Gap Condition: The expansion of $(-5i)^{\chi}G^{\chi}$ around the conifold ($s_0=\mu s_4$, i.e., $z=5^{-5}$) is of the form 
    
\begin{equation}\label{eq:gap}
\Phi_\chi t_c^{-\chi} + O(1),\quad
{\Phi}_\chi = (-5)^{-\chi/2}\Big(c_\chi+r_\chi\Big),
\end{equation}

\begin{equation}\label{eq:KR}
c_\chi=\frac{B_{\chi+2}}{2\chi(\chi+2)},
\qquad
r_\chi=\frac{(\chi+2)E_{\chi}-2\left(2^{\chi+2}-1\right)B_{\chi+2}}
                     {2^{\chi+3}\,\chi(\chi+2)},
\end{equation}
 and $B_\chi$ and $E_\chi$ are Bernoulli numbers and Euler numbers. The contribution $\mathcal K$ is the contribution from the complex case (closed/orientable), while the contribution $r_{\chi}$ comes from the real (open/non orientable)
    \item The expansion of $(-5i)^{\chi}G^{\chi}$ around the holomorphic limit ($s_4=0$, $z=0$) has constant term equals zero for odd $\chi$ and given by 
    
\begin{equation}\label{eq:const-map}
m_\chi \;=\; \frac{(-1)^{\chi/2}\,B_{\chi+2}B_{\chi}}{4\,\chi(\chi+2)\,\chi!}\,e(X)
\;=\;\frac{(-1)^{\chi/2+1}\,50\,B_{\chi+2}B_{\chi}}{\chi(\chi+2)\,\chi!}\ ,
\ \ e(X)=-200 .
\end{equation}
for even $\chi$. This corresponds to the contribution of constant maps, which only appear at the closed amplitudes. Heuristically, the constant term should be zero because the contribution of constant maps is proportional to the Euler characteristic, which is zero for a real threefold.
    \item Vanishing of Complex Invariants: real Gopakumar-Vafa invariants are expected to be zero whenever their complex counterparts are zero. The idea is if there are no complex curves, there can be no real curves. We see this being used in~\cite{walcher_evidence_2009} and~\cite{krefl_real_2008}. In~\cite{piazzalunga_m-theory_2014}, the authors give an interpretation of this statement in terms of the BPS states. We use this specially for low degree: for a fixed genus, there is a minimum degree for which we have curves in $X$. We can cite, in particular, the Castelnuovo bound~\cite{liu_castelnuovo_2022}. In~\cite{huang_topological_2009}, those bounds were used to obtain extra conditions to fix the ambiguities for the complex invariants. 
    \begin{equation}\label{eq:vanishing}
        n^{\mathbb C}_{g,d}=0\implies n^{\mathbb R}_{g,d}=0
    \end{equation}

    We also have that real invariants agree with complex invariants modulo 2\cite{walcher_evidence_2009,krefl_real_2008}. This is not used to fix the ambiguities, but we checked that the numbers we obtained satisfy that condition. The

\end{enumerate}
\begin{remark}\label{rem:KW}
    We should explain where the $r_\chi$ part of the gap condition came from. In~\cite{krefl_real_2009}, values werre proposed for lower $\chi$, after fixing lower genus ambiguities for local $\mathbb P^2$. The full formula was proposed in~\cite[Eq. (5.6)]{krefl_extended_2011}. 
    
    Comparing the formulas, there are two details in our computation that the reader should pay attention to: the first is that we added the normalization $(-5)^{-\chi/2}$ that correspond to a change of coordinates $t_c\to i\sqrt{5} t_c$.
     The second is the division by two, coming from the choice of normalization in equation~\eqref{eq:G-decomp-norm}.
\end{remark}
\subsection{Algorithm and Computation}\label{sec:amb-alg-com}
The expression~\eqref{eq:amb-ansatz-final} gives us $\left\lfloor\frac{3\chi}2\right\rfloor+1$ coefficients to fix. The gap condition fixes $\chi$ coefficients and the constant term (which appears only for $\chi$ even), fixes one more. We have $\frac\chi2$ coefficients left for even $\chi$ and $\frac{\chi+1}2$ left for odd $\chi$. The vanishing of the complex invariants allows us to fix further coefficients, assuming all conditions are independent.

The important insight here is that, since we have only half the coefficients $(\chi\equiv_2k)$ on the ambiguity, we do not have as many invariants as it could seem at first. We can fix the ambiguity until $\chi=4$ and, once we have the GV invariant $n_{6,5}^{\mathbb R}$, until $\chi=7$.  See table~\ref{tab:count-quintic}.

\begin{table}[ht]
\centering
\small
\resizebox{0.8\textwidth}{!}{
\begin{tabular}{cccccccl}
\toprule
$\chi$ & $g$ & $\#\,b_k$ & conifold & const.\ map & degrees $d$ with $n^{\mathbb C}_{g,d}=0$ & $\#$ & remaining \\
\midrule
$1$ & $2$ & $2$ & $1$ & -- & $1,3$ & $2$ & $-1$ \\
$2$ & $3$ & $4$ & $2$ & $1$ & $2$ & $1$ & $0$ \\
$3$ & $4$ & $5$ & $3$ & -- & $1,3$ & $2$ & $0$ \\
$4$ & $5$ & $7$ & $4$ & $1$ & $2,4$ & $2$ & $0$ \\
$5$ & $6$ & $8$ & $5$ & -- & $1,3$ & $2$ & $1$: $n^{\mathbb R}_{6,5}$ \\
$6$ & $7$ & $10$ & $6$ & $1$ & $2,4,6$ & $3$ & $0$ \\
$7$ & $8$ & $11$ & $7$ & -- & $1,3,5,7$ & $4$ & $0$ \\
$8$ & $9$ & $13$ & $8$ & $1$ & $2,4,6$ & $3$ & $1$ \\
$9$ & $10$ & $14$ & $9$ & -- & $1,3,5,7$ & $4$ & $1$ \\
$10$ & $11$ & $16$ & $10$ & $1$ & $2,4,6,8$ & $4$ & $1$ \\
$11$ & $12$ & $17$ & $11$ & -- & $1,3,5,7$ & $4$ & $2$ \\
$12$ & $13$ & $19$ & $12$ & $1$ & $2,4,6,8$ & $4$ & $2$ \\
\bottomrule
\end{tabular}}
\caption{Conditions on the ambiguity~\eqref{eq:amb-ansatz-final} of the mirror quintic, $g=\chi+1$.}%
\label{tab:count-quintic}
\end{table}

The algorithm is the following, basically the same as the Huang--Klemm--Quackenbush \cite{huang_topological_2009} for the closed case. 

1. Solve the algebraic equation from Theorem~\ref{alg-HAE-GMCD} by partial integration starting from the lower $\chi$ generating functions.

2. Add the ambiguity $s_0^{
2\chi}h^{\chi}$ following the equation in Proposition~\ref{amb-ansatz}. 

3. Consider the expansion around the conifold. Each variable $s_i$ has an expansion around the conifold that can be calculated from the vector field ${\sf R}_0$ and initial conditions. Use the gap condition to compute $\chi$ terms $b_k$.

4. Compare the expansion in the holomorphic limit with the BPS expansion. The vanishing of the complex invariants~\eqref{eq:vanishing} gives a way of computing more coefficients. 

We write the expression of $G^{\chi}$ in terms of $s_i$ for $\chi=1$ and the BPS invariants for genus $g=2$ to $g=5$. We reproduce Walcher's results for $\chi=1$ and, for higher genus, we obtain integrality of the BPS invariants and they agree with the complex ones modulo 2, which is evidence that they are the correct numbers.

\subsection{Beyond \texorpdfstring{$\chi=4$}{4} and geometric hints}\label{sec:amb-geometric}
For $\chi=5$, we are short one condition to compute all the coefficients, and that is exactly the invariant $n_{6,5}^{\mathbb R}$. To compute the invariant $n^{\mathbb R}_{6,5}$, we proceed as in the complex case, where one identifies the space of genus 6 curves of degree 5 with the space of plane sections of $X$, which is simply the Grassmannian $G(3,5)$ of $\mathbb P^2$'s in $\mathbb P^4$ and compute its Euler characteristic (Donaldson-Thomas invariant). The real version should be simply the Euler characteristic of the real Grassmannian, which is 2. Similar ideas were considered in~\cite{espreafico_motivic_2024}, where arithmetic versions of DT invariants were studied.

The sign, however is not clear to us. Following the computation of~\cite{espreafico_motivic_2024}, the sign should be $(-1)^{\frac{\dim G(3,5)}2}=-1$. It is important to mention that neither the integrality property nor the fact that it has to be even (since the complex number is 10) can fix the invariant. Any even integer would work. In Appendix~\ref{app:tables}, we list the invariants for $\chi\le 4$ and, for both choices of sign, for $\chi=5,6,7$. 

Some other numbers can be geometrically justified. We take, for instance $n_{3,4}^{\R} = 15$. Any curve of degree four and genus three spans a plane $P$, since $g=(4-1)(4-2)/2 = 3$. If the curve is on the quintic threefold $X$, the intersection $P\cap X$ is the curve and a line (since the degree is five). Therefore, for each line on the quintic threefold, we can consider the space of all planes $P$ through it. The intersection $X\cap P$ is a plane curve of degree 4 in $P$ intersecting a line. The space of planes containing a line is simply the Grassmannian $G(1,3)\cong\mathbb P^2$. Therefore, the Gopakumar-Vafa invariant should be the Euler characteristic of the space of such curves, which is $\mathbb P^2$ times the space of lines. It will give $\chi(\text{ lines in $X$})\chi(\mathbb P^2)=2875\cdot 3 = 8625$, over $\mathbb C$. Taking the real Euler characteristic, we count lines with signs and use $\chi(\mathbb P^2)=1$ to obtain $15\cdot 1= 15$. Note that the sign convention mentioned above does not apply in this case.

\subsection{Computations for \texorpdfstring{$\chi=1$}{1}}\label{sec:chi12}

In this section we carry out the algorithm above for $\chi=1$. We solve the anomaly equation, fix the ambiguity and compare it with the BPS expansion.

The same process can be repeated for $\chi=2$, but the equations get very large quickly, so we opted to omit the explicit computation.

\subsubsection*{The case $\chi=1$}
For $\chi=1$, Theorem~\ref{alg-HAE-GMCD} reads
\begin{align}
    {\sf R}_1G^1={\sf R}_3G^1&=0, & {\sf R}_2G^1&=G^1,\\
    {\sf R}_4G^1&=\frac1{5^8}{\sf R}_0^2G^{-1}=\frac{\sf F}{2\cdot 5^8}, & {\sf R}_7G^1&=-\frac{1}{5^8}{\sf R}_0G^{0}.
\end{align}
Using the expressions for ${\sf R}_0G^0$ and ${\sf R}_0^2G^{-1}$ given in Theorem~\ref{alg-HAE-GMCD} and Theorem~\ref{partial-int}, we can calculate all derivatives $\partial_{s_i}G^1$ for $i\ne0,4$. Integrating directly, we get
\begin{equation}\label{eq:G1}
G^1=\frac{1}{5^8}\left(-s_7\,{\sf R}_0G^0+\frac{\dis^2s_7^3}{6s_5^3}-\frac{25}{6}\,s_8\right)+\frac{b_1s_4^5}{s_0^3}+\frac{b_3s_4^{15}}{s_0^3\,\dis},
\end{equation}
that is,
\begin{multline*}
G^1=\frac{1}{48\cdot 5^8\,s_0^2s_5^3}\Big(\big(5^4(1022s_0^{10}+3s_4^{10})+200s_0^2s_3\big)s_5^2s_7-15000\,\dis\, s_0^2s_5s_6s_7\\
-4\,\dis^2s_0^2s_7^3-200s_0^2s_5^3s_8\Big)+\frac{b_1s_4^5}{s_0^3}+\frac{b_3s_4^{15}}{s_0^3\,\dis}.
\end{multline*}
The last two terms are the ambiguity of Proposition~\ref{amb-ansatz} ($k=1,3$).

To compute $b_1$ and $b_3$, we expand around the conifold by solving ${\sf R}_0=5\frac{d}{dt_c}$ with the conifold initial conditions of Proposition~\ref{prop:expansions}. Replacing the $s_i$ by these solutions, we get
\begin{equation}
(-5i)\,G^1=-\frac{i\,b_3}{t_c}-i\Big(b_1-\frac45 b_3\Big)+O(t_c).
\end{equation}
Since $\chi$ is odd, the gap condition says that $G^1$ is regular at the conifold, so $b_3=0$. Around $z=0$, solving ${\sf R}_0=5q\frac{d}{dq}$ with the initial conditions of Proposition~\ref{prop:expansions} gives
\begin{equation}
\begin{aligned}
(-5i)\,G^1={}&-\frac{25i}{16}\big(16\sqrt5\,b_1+3\big)q^{1/2}\\
&+\frac{5i}{16}\big(21200\sqrt5\,b_1-250000\sqrt5\,b_3+4609\big)q^{3/2}+O(q^{5/2}).
\end{aligned}
\end{equation}
Comparing with the BPS expansion~\ref{eq:BPS-real}, we have:
\begin{equation}
\begin{aligned}
(-5i)\,G^1&=-i\sum_{d\ {\rm odd}}\ \sum_{k\ {\rm odd}}\Big(n^{\mathbb R}_{2,d}-\frac{1}{24}\,n^{\mathbb R}_{0,d}\Big)q^{kd/2}\\
&=i\Big(\frac58-n^{\mathbb R}_{2,1}\Big)q^{1/2}+i\Big(\frac{65}{2}-n^{\mathbb R}_{2,1}-n^{\mathbb R}_{2,3}\Big)q^{3/2}+O(q^{5/2}),
\end{aligned}
\end{equation}
where we used the disk invariants $n^{\mathbb R}_{0,1}=15$ and $n^{\mathbb R}_{0,3}=765$. The next condition comes from the vanishing of the complex invariants: $n^{\mathbb C}_{2,1}=0$ gives $n^{\mathbb R}_{2,1}=0$. This is enough to fix

\begin{equation}\label{eq:b-chi1}
b_1=-\frac{17}{80\sqrt5},\qquad b_3=0 .
\end{equation}

The coefficient of $q^{3/2}$ then gives $n^{\mathbb R}_{2,3}=0$, which is exactly what the vanishing of the complex invariants predicts, since $n^{\mathbb C}_{2,3}=0$; this is a first consistency check. The higher coefficients give

\begin{equation}\label{eq:nR-chi1}
\begin{gathered}
n^{\mathbb R}_{2,5}=-55640,\qquad n^{\mathbb R}_{2,7}=-159440655,\qquad n^{\mathbb R}_{2,9}=-387012696805,\\
n^{\mathbb R}_{2,11}=-878665820903170,\qquad n^{\mathbb R}_{2,13}=-1923803069641669385,\ \dots
\end{gathered}
\end{equation}

Notice that we get the exact same values as~\cite{walcher_evidence_2009}.

\section{Conclusions and Generalizations}\label{sec:conclusions}

Although our work focused on the quintic, the same methods apply to other models. We implemented the same computation for the bicubic, as it was the other example in~\cite{walcher_evidence_2009}. The only thing that is lacking for the bicubic is an algebraic counterpart of the normal function (the pair of conics for the quintic), so we can't define a relative cohomology as we did for the quintic. This can be remedied by simply working in an abstract mixed Hodge structure of dimension five, where four of the generators form $H^3_{dR}(Y_z)$ and the extra generator correspond to the normal function is on $F^1$ and $W_2$, as in the quintic case. 

\subsection{GMCD and Anomaly Equations for the Bicubic}

The bicubic in $\mathbb P^5$ is the complete intersection of two cubic equations. It is a complete intersection Calabi-Yau threefold and its mirror was studied in~\cite{libgober_lines_1993}. It is given by the resolution of singularities of the quotient of the intersection of two families of cubic hypersurfaces in $\mathbb P^5$:
\begin{align}
    x_1^3+x_2^3+x_3^3-3\lambda x_4x_5x_6 = 0\\
    x_4^3+x_5^3+x_6^3-3\lambda x_1x_2x_3 = 0
\end{align}
by the group $G_{81}\subset\Z_9\times\Z_3^2\times\Z_3^2$ of elements $(\mu,\alpha,\beta,\gamma,\delta)$ with $3\mu=\alpha+\beta=\gamma+\delta$ that acts as $\left(\zeta_3^{\alpha}\cdot \zeta_9^{\mu}\cdot x_1,\ \zeta_3^{\beta}\cdot \zeta_9^{\mu}\cdot x_2,\ \zeta_9^{\mu}\cdot x_3,\ \zeta_3^{-\delta}\cdot \zeta_9^{-\mu}\cdot x_4,\ \zeta_3^{-\epsilon}\cdot \zeta_9^{-\mu}\cdot x_5,\ \zeta_9^{-\mu}\cdot x_6\right)$, where $\zeta_k^k=1$.
\smallskip

The periods satisfy a Picard-Fuchs equation 
\begin{equation}\label{eq:PF-bicubic}
  \mathcal D=\theta^4-9z\,(3\theta+1)^2(3\theta+2)^2
  =\theta^4-3^6 z\left(\theta+\tfrac13\right)^2\left(\theta+\tfrac23\right)^2,
  \qquad \theta=z\frac{d}{dz},
\end{equation}
and the extra real period $\tau$ satisfy:
\begin{equation}\label{eq:ext-PF-bicubic}
    \mathcal D \tau = \frac 98\sqrt{z}
\end{equation}

Solving it gives a basis of four closed periods, corresponding to a basis of $H^3_{\rm dR}(Y_z)$ and one extra period. We do not write the expressions of those solutions, but they can be found together with the codes accompanying this paper. We just write the mirror map and the expression for $\tau/\varpi_0$, which gives the real Gromov-Witten invariants of genus zero after dividing by two and removing multiple coverings.
\begin{align}\label{eq:mirror-bicubic}
  q(z)&=z+180\,z^2+55890\,z^3+22089840\,z^4+\dots,\\
  \frac{\tau}{\varpi_0}&=18\,q^{1/2}+182\,q^{3/2}+\tfrac{787968}{25}\,q^{5/2}+\tfrac{323202744}{49}\,q^{7/2}+\dots
\end{align}

\begin{definition}\label{def:enhanced-bicubic}
    The space of relatively enhanced mirror bicubics is the space of pairs $(Y_z,\alpha=\{\alpha_0,\alpha_1,\dots,\alpha_4\})$, where $\alpha$ form a basis of $V=\langle\alpha_0\rangle\oplus H^3(Y_z)$ for which $\alpha_i$, $i=1,\dots4$ respect the Hodge structure and have constant intersection product. We define a mixed Hodge structure on $V$ by making $\alpha_0$ generate $W_2$ and $\alpha_0\in F^1$.
\end{definition}

Such space is going to have dimension 9, as in the quintic case and the Gauss-Manin connection can be computed in exactly the same manner, using that it should commute with the normal function associated to the extra period (since it corresponds to the integral over a holomorphic curve). The theorems below are analogous to the ones for the quintic.

\begin{theorem}\label{thm:bicubic-moduli}
    The moduli space of relatively enhanced mirror bicubics is 9 dimensional and given by
\begin{equation}\label{eq:S-bicubic}
\rT={\rm Spec}\,\mathbb C\left[s_0,s_1,\dots,s_8,\frac{1}{s_0\,s_4\,s_5\left(s_0^{6}-s_4^{6}\right)}\right],\qquad z=\frac{s_4^{6}}{3^{6}s_0^{6}}.
\end{equation}
\end{theorem}

\begin{theorem}\label{thm:bicubic-R0}
    There are unique $\sf R_0$, $\sf Y$ and $\sf F$ such that equation~\ref{eq:GM-R0} is satisfied. Those are given below.
{\footnotesize
\begin{equation*}\label{eq:R0-bicubic}
{\sf R}_0:\left\{\begin{aligned}
\dot s_0 &= \frac{s_0\left(2\cdot3^{4}s_0^{8}-3^{4}s_0^{2}s_4^{6}+s_3\right)}{2s_5}\\
\dot s_1 &= \frac{-3^{6}s_0^{12}+5\cdot3^{3}s_0^{8}s_1+3^{6}s_0^{6}s_4^{6}-2\cdot3^{3}s_0^{2}s_1s_4^{6}+s_1s_3}{s_5}\\
\dot s_2 &= \frac{-5\cdot3^{7}s_0^{14}-3^{3}s_0^{10}s_1+10\cdot3^{3}s_0^{8}s_2+5\cdot3^{7}s_0^{8}s_4^{6}+3^{3}s_0^{4}s_1s_4^{6}-4\cdot3^{3}s_0^{2}s_2s_4^{6}+2s_2s_3}{s_5}\\
\dot s_3 &= \frac{-25\cdot3^{6}s_0^{16}-3^{3}s_0^{10}s_2+5\cdot3^{4}s_0^{8}s_3+3^{9}s_0^{10}s_4^{6}+3^{3}s_0^{4}s_2s_4^{6}-2\cdot3^{4}s_0^{2}s_3s_4^{6}-2\cdot3^{6}s_0^{4}s_4^{12}+3s_3^{2}}{s_5}\\
\dot s_4 &= \frac{s_4\left(5\cdot3^{3}s_0^{8}-2\cdot3^{3}s_0^{2}s_4^{6}+s_3\right)}{2s_5}\\
\dot s_5 &= \frac{5\cdot3^{4}s_0^{8}s_5-2\cdot3^{4}s_0^{2}s_4^{6}s_5+2s_3s_5-3^{3}s_0^{4}\left(s_0^{6}-s_4^{6}\right)s_6}{s_5}\\
\dot s_6 &= \frac{-5\cdot3^{3}s_0^{6}s_5-2s_2s_5+5\cdot3^{4}s_0^{8}s_6-2\cdot3^{4}s_0^{2}s_4^{6}s_6+3s_3s_6}{s_5}\\
\dot s_7 &= -s_8\\
\dot s_8 &= \frac{3^{9}s_0s_4^{3}\left(s_0^{6}-s_4^{6}\right)}{8s_5}
\end{aligned}\right.
\end{equation*}
}
\vspace{-0.5cm}
\begin{equation}\label{eq:YF-bicubic}
{\sf Y}=\frac{3^{6}s_0^{8}\left(s_0^{6}-s_4^{6}\right)^{2}}{s_5^{3}},\quad {\sf F}=-{\sf Y}s_7
\end{equation}
\end{theorem}
The solutions to $\sf R_0$ for specific initial conditions give us the expansions at the conifold $(z=3^{-6})$ and at the large complex structure ($z=0$). Continuing the same programme, we can compute $\rG$, its Lie algebra and the vector fields $\sf R_i$. We can also derive a similar dictionary as~\ref{tab:dictionary}, given in table~\ref{tab:dictionary-bicubic}.

\begin{table}[H]
\centering\small
\renewcommand{\arraystretch}{1.7}
\resizebox{0.95\textwidth}{!}{%
\begin{tabular}{@{}lc@{}}
\toprule
special geometry on $\mathcal M$ & $\beta\in\mathcal O_{\rT}$ with $\beta\circ\sigma=$ left column \\
\midrule
$e^{-K}$ & $s_0^{2}$ \\
$z$ & $\frac{s_4^{6}}{3^{6}s_0^{6}}$ \\
$G_{z\bar z}$ & $\hat G_{z\bar z}=-\frac{3^{3}s_0^{4}s_5}{s_4^{6}\,(s_0^{6}-s_4^{6})}$ \\
$K_z$ & $\hat K_z=\frac{3^{2}s_0^{4}\left(3^{4}s_0^{2}(2s_0^{6}-s_4^{6})+s_3\right)}{s_4^{6}\,(s_0^{6}-s_4^{6})}$ \\
$S^{zz}$ & $\hat S^{zz}=\frac{s_4^{12}\,(s_0^{6}-s_4^{6})\,(3s_5-s_0^{2}s_6)}{3^{15}\,s_0^{18}s_5}$ \\
$S^{z}$ & $\hat S^{z}=\frac{s_4^{6}\bigl(3^{3}s_0^{2}(11s_0^{6}-9s_4^{6})+3s_3-s_0^{2}s_2\bigr)}{3^{13}\,s_0^{14}}$ \\
$S$ & $\hat S=\frac{3^{7}s_0^{4}(s_0^{6}-s_4^{6})(20s_0^{6}-11s_4^{6})+3^{5}s_0^{2}s_3(3s_0^{6}-2s_4^{6})+3s_3^{2}-s_0^{2}s_2\bigl(s_3+3^{3}s_0^{2}(5s_0^{6}-2s_4^{6})\bigr)-3^{3}s_0^{6}s_1(s_0^{6}-s_4^{6})}{2\cdot 3^{11}\,s_0^{10}\,(s_0^{6}-s_4^{6})}$ \\
$\Delta^{z}$ & $\hat\Delta^{z}=\frac{i\,s_4^{6}\,(s_0^{6}-s_4^{6})\,s_7}{3^{10}\,s_0^{6}s_5}$ \\
$\Delta$ & $\hat\Delta=\frac{i\,s_8}{3^{8}\,s_0^{2}}$ \\
\bottomrule
\end{tabular}%
    }
\caption{The generators of the polynomial ring from~\cite{alim_polynomial_2007} in terms of the $s_i$, for the mirror bicubic.}\label{tab:dictionary-bicubic}
\end{table}

\begin{theorem}
    The holomorphic anomaly equation for the bicubic can be written in the following form:
\begin{equation}\label{eq:HAE-bicubic}
\begin{gathered}
{\sf R}_1G^{\chi}={\sf R}_3G^{\chi}=0,\qquad {\sf R}_2G^{\chi}=\chi G^{\chi},\\
{\sf R}_4G^{\chi}=\frac{1}{3^{8}}\left(\sum_{\substack{\chi_1+\chi_2=\chi-2\\ \chi_1,\chi_2\geq0}}{\sf R}_0G^{\chi_1}{\sf R}_0G^{\chi_2}+{\sf R}_0^{2}G^{\chi-2}\right),\qquad
{\sf R}_7G^{\chi}=-\frac{1}{3^{8}}{\sf R}_0G^{\chi-1},
\end{gathered}
\end{equation}
with $G^{\chi}=i^{\chi}s_0^{2\chi}\mathcal G^{\chi}$,
\begin{equation}\label{eq:seeds-bicubic}
\begin{gathered}
{\sf R}_0G^{0}=\frac{s_0^{8}\left(s_0^{6}-s_4^{6}\right)^{2}s_7^{2}}{36\,s_5^{3}}+\frac{3^{3}s_0^{4}\left(s_0^{6}-s_4^{6}\right)s_6}{2s_5^{2}}-\frac{3\left(2322\,s_0^{8}-981\,s_0^{2}s_4^{6}+16\,s_3\right)}{16\,s_5},\\
{\sf R}_0^{2}G^{-1}=\tfrac12{\sf F},
\end{gathered}
\end{equation}
where ${\sf R}_1,{\sf R}_2,{\sf R}_6,{\sf R}_7,{\sf R}_8$ are given by the same formulas as for the mirror quintic and
\begin{equation}\label{eq:R345-bicubic}
\begin{aligned}
{\sf R}_3&=\frac{s_5\left(5\cdot3^{3}s_0^{6}+s_2\right)-s_6\left(5\cdot3^{3}s_0^{8}-2\cdot3^{3}s_0^{2}s_4^{6}+s_3\right)}{3^{3}s_0^{4}\left(s_0^{6}-s_4^{6}\right)}\frac{\partial}{\partial s_1}+s_6\frac{\partial}{\partial s_2}+s_5\frac{\partial}{\partial s_3}+s_7\frac{\partial}{\partial s_8},\\
{\sf R}_4&=\frac{3^{3}s_0^{4}\left(s_0^{6}-s_4^{6}\right)}{s_5}\frac{\partial}{\partial s_6},\\
{\sf R}_5&=-\frac{5\cdot3^{3}s_0^{8}-2\cdot3^{3}s_0^{2}s_4^{6}+s_3}{s_5}\frac{\partial}{\partial s_1}+\frac{3^{3}s_0^{4}\left(s_0^{6}-s_4^{6}\right)}{s_5}\frac{\partial}{\partial s_2}.
\end{aligned}
\end{equation}
\end{theorem}
Solving the equation, we get regular functions on the moduli space and the ambiguity has the form (with the standard assumptions and BPS form~\eqref{eq:BPS-real}):
\begin{equation}\label{eq:amb-bicubic}
s_0^{2\chi}h^{\chi}(s_0,s_4),\qquad h^{\chi}(s_0,s_4)=\frac{1}{s_0^{3\chi}}\sum_{\substack{0\le k\le 3\chi\\ k\equiv\chi\ (2)}}b_k\,s_4^{3k}\left(s_0^{6}-s_4^{6}\right)^{\frac{\chi-k}{2}}.
\end{equation}

We then fix the ambiguity with the same strategy: the conifold gives us $\chi$ coefficients and the constant term gives one more if $\chi$ is even. The conifold and constant term conditions are the same as for the quintic, but the normalization involves $(-3i)^\chi$ and we have $e(X)=-144$. We fix the rest of coefficients using the fact that the vanishing of the complex invariants implies the vanishing of the real invariants. 

That allows us to compute GV up $\chi = 8$. All the numbers obtained are integer and agree with the complex ones mod 2 (see Tables~\ref{tab:bicubic-gv-01}--\ref{tab:bicubic-gv-89} in Appendix~\ref{app:tables}). For $\chi=9$, we are short one condition and the invariant $n_{10,9}^{\mathbb R}$ has to be determined. 

In the complex case, since the genus of the intersection of two cubic surfaces in $\mathbb P^3$ (i.e., 3-plane section of the bicubic) has genus 10. That means that the space of such curves is a Grassmannian $G(4,6)$ and the invariant $n_{10,9}$ is the Euler characteristic $\chi(G(4,6)) = 15$. Using the same type of prediction as we used for the quintic, we can take the Euler characteristic of the real Grassmannian to obtain $n_{10,9}^{\mathbb R} = \pm 3$. Following the signs from~\cite{espreafico_motivic_2024}, we would have $(-1)^{8/2}=1$, i.e., $n_{10,9}^{\mathbb R}=+3$. Again, integrality and the fact the invariants agree modulo 2 are not enough to fix the invariants.

Using that prediction, we can go up to $\chi=11$, but no further, unless we can get more geometric data from the $A$-model or conditions for the the expansions around $z=\infty$ See the table~\ref{tab:count-bicubic} below.

\begin{table}[ht!]
\centering
\small
\begin{tabular}{cccccccl}
\toprule
$\chi$ & $g$ & $\#\,b_k$ & conifold & const.\ map & degrees $d$ with $n^{\mathbb C}_{g,d}=0$ & $\#$ & remaining \\
\midrule
$1$ & $2$ & $2$ & $1$ & -- & $1,3$ & $2$ & $-1$ (check) \\
$2$ & $3$ & $4$ & $2$ & $1$ & $2,4$ & $2$ & $-1$ (check) \\
$3$ & $4$ & $5$ & $3$ & -- & $1,3,5$ & $3$ & $-1$ (check) \\
$4$ & $5$ & $7$ & $4$ & $1$ & $2,4,6$ & $3$ & $-1$ (check) \\
$5$ & $6$ & $8$ & $5$ & -- & $1,3,5,7$ & $4$ & $-1$ (check) \\
$6$ & $7$ & $10$ & $6$ & $1$ & $2,4,6$ & $3$ & $0$ \\
$7$ & $8$ & $11$ & $7$ & -- & $1,3,5,7$ & $4$ & $0$ \\
$8$ & $9$ & $13$ & $8$ & $1$ & $2,4,6,8$ & $4$ & $0$ \\
$9$ & $10$ & $14$ & $9$ & -- & $1,3,5,7$ & $4$ & $1$: $n^{\mathbb R}_{10,9}$ \\
$10$ & $11$ & $16$ & $10$ & $1$ & $2,4,6,8,10$ & $5$ & $0$ \\
$11$ & $12$ & $17$ & $11$ & -- & $1,3,5,7,9,11$ & $6$ & $0$ \\
$12$ & $13$ & $19$ & $12$ & $1$ & $2,4,6,8,10$ & $5$ & $1$ \\
\bottomrule
\end{tabular}
\caption{Number of coefficients fixed by each boundary condition for the ambiguity~\eqref{eq:amb-bicubic} of the bicubic}%
\label{tab:count-bicubic}
\end{table}

\subsection{ Further problems}

This paper uses the Gauss-Manin connection in Disguise framework to solve the extended holomorphic anomaly equation and then uses the conifold gap condition and the vanishing of complex invariants to further fix ambiguities and compute predictions for Gopakumar-Vafa invariants. It leaves however, several questions unaddressed. 

\begin{enumerate}
    \item A-model confirmation. The numbers computed here are simply predictions of real Gopakumar-Vafa invariants, as they have not been rigorously proven to correspond to any geometry on the A-model. However, following the ideas from the proofs for the closed invariants (\cite{chang_bcov_2026,chang_conifold_2026,guo_structure_2018}) it might be possible to prove similar results for the real case. 
    
    \item Further conditions. In order to fix the ambiguities of the closed amplitudes, we have to use that the expansions around $z=\infty$ are regular. This doesn't happen for the real amplitudes. It would be interesting to find whether there is any information that could be extracted from the behavior around this point.
    
    \item Quasimodularity. In section~\ref{sec: periods}, we mention that the variable $\tau_0$ has a modular behavior with respect to the monodromy group, since it acts on the basis $\delta$ and therefore changes $\tau_0$. This type of modularity (also considered in the closed case by~\cite{movasati_modular-type_2015}) although analogous to quasimodular forms (the vector field $\sf R_0$ plays the role of the Ramanujan vector field) is not known to correspond to any special functions. 

    \item Confirmation for the values $n^{\mathbb R}_{6,5}$ and other values at the Castelnuovo bound~\cite{liu_castelnuovo_2022}. The prediction from~\cite{espreafico_motivic_2024} goes in this direction, but a proper computation of Gromov-Witten invariants has not been carried out.
    \item The AMSY Lie algebra in the A-model. It is worth studying to which object the AMSY Lie algebra corresponds in the A-model. Since it encodes the change of basis of $H^3_{\rm dR}(Y_z)$, it should be related to some structure in the quantum cohomology of $X$. 
    
\end{enumerate}

\newpage
\appendix
\section{Tables of Real GV invariants}\label{app:tables}

The invariants of real genus $g=\chi+1\le 5$ for the quintic and $g=\chi+1\le9$ for the bicubic are fixed by the conditions of Section~\ref{sec:ambiguity} alone. For $g=6,7,8$ for the quintic and $g=10,11,12$ for the bicubic we list the numbers for the choices $n^{\mathbb R}_{6,5}=\pm2$ and $n^{\mathbb R}_{10,9}=\pm3$. All numbers are integers, vanish whenever $d+g$ is even, and agree with the complex invariants $n^{\mathbb C}_{g,d}$ modulo $2$. The genus $0,1,2$ columns agree with~\cite[Tables 5--7]{walcher_evidence_2009}.

\begin{table}[h]
\centering
\scriptsize
\setlength{\tabcolsep}{2.5pt}
\renewcommand{\arraystretch}{0.95}
\begin{tabular}{rr@{\hspace{1.2em}}rr}
\toprule
$d$ & $n^{\mathbb R}_{0,d}$ & $d$ & $n^{\mathbb R}_{1,d}$ \\
\midrule
1 & 15 & 2 & 0 \\
3 & 765 & 4 & 17725 \\
5 & 544125 & 6 & 36079420 \\
7 & 487998390 & 8 & 65378348025 \\
9 & 536543881350 & 10 & 116755627418596 \\
11 & 664513551962205 & 12 & 209184366237053675 \\
13 & 890983311920874465 & 14 & 377367159810138778200 \\
15 & 1264123608455988294750 & 16 & 685680464985559781173935 \\
17 & 1871028346129178222325990 & 18 & 1254195667395460913669340460 \\
19 & 2861726040699237604475400135 & 20 & 2307695080571694537792278553240 \\
21 & 4493230049007630091514258681445 & 22 & 4268294781233155927303999369760000 \\
23 & 7207522320216113436677394290218890 & 24 & 7930864794860957117764374341488743775 \\
25 & 11769233993986933173028634425462458750 & 26 & 14795914375403697763257302788288699236460 \\
27 & 19509482253918078839431328789761148877375 & 28 & 27702421548882194251036108416804908638214050 \\
29 & 32759521641137690288799693558477149120583585 & 30 & 52032854441736087682542435669500514830832641768 \\
\bottomrule
\end{tabular}
\caption{Real Gopakumar--Vafa invariants of the quintic, $g=0,1$ ($\chi=-1,0$), $d\le 30$.}%
\label{tab:quintic-gv-01}
\end{table}

\begin{table}[h]
\centering
\scriptsize
\setlength{\tabcolsep}{2.5pt}
\renewcommand{\arraystretch}{0.95}
\begin{tabular}{rr@{\hspace{1.2em}}rr}
\toprule
$d$ & $n^{\mathbb R}_{2,d}$ & $d$ & $n^{\mathbb R}_{3,d}$ \\
\midrule
1 & 0 & 2 & 0 \\
3 & 0 & 4 & 15 \\
5 & $-55640$ & 6 & $-207405$ \\
7 & $-159440655$ & 8 & $-4619005180$ \\
9 & $-387012696805$ & 10 & $-27137742347875$ \\
11 & $-878665820903170$ & 12 & $-108886944055496170$ \\
13 & $-1923803069641669385$ & 14 & $-362235070015820431070$ \\
15 & $-4118149267275418609230$ & 16 & $-1078138887521929667365285$ \\
17 & $-8682546731231153534612880$ & 18 & $-2981711892726467319106523605$ \\
19 & $-18109093970109245075895234500$ & 20 & $-7829093080747030560785379434175$ \\
21 & $-37468264493303566297526373243240$ & 22 & $-19779824170903622987204843075280500$ \\
23 & $-77048313163738850395654093685360465$ & 24 & $-48511885003920820556658437675177155180$ \\
25 & $-157676389901970679932873142860244857050$ & 26 & $-116215863547122205713863347058482398407325$ \\
27 & $-321431066325254450956025609607695714330450$ & 28 & $-273154022420027278567312773236139254096336875$ \\
29 & $-653182477927581090668112378964111073302057375$ & 30 & $-631996528305705272138475935490554475556077070350$ \\
\bottomrule
\end{tabular}
\caption{Real Gopakumar--Vafa invariants of the quintic, $g=2,3$ ($\chi=1,2$), $d\le 30$.}%
\label{tab:quintic-gv-23}
\end{table}

\begin{table}[ht]
\centering
\scriptsize
\setlength{\tabcolsep}{2.5pt}
\renewcommand{\arraystretch}{0.95}
\begin{tabular}{rr@{\hspace{1.2em}}rr}
\toprule
$d$ & $n^{\mathbb R}_{4,d}$ & $d$ & $n^{\mathbb R}_{5,d}$ \\
\midrule
1 & 0 & 2 & 0 \\
3 & 0 & 4 & 0 \\
5 & $-90$ & 6 & $-15$ \\
7 & 1234970 & 8 & 52170905 \\
9 & 35788777610 & 10 & 1905593363625 \\
11 & 245624090387605 & 12 & 21681640787251870 \\
13 & 1133471895851344625 & 14 & 150488615238868518655 \\
15 & 4277579205095611088195 & 16 & 791453729946273242589215 \\
17 & 14268856132624358361630530 & 18 & 3483363440610165211459444580 \\
19 & 43768272540752014024081660575 & 20 & 13542582087751486390324942910400 \\
21 & 126333409831717590697685306283305 & 22 & 48055969739730032337427038541278275 \\
23 & 348203669368376202626716715601122470 & 24 & 159000242923651516015027691220348626435 \\
25 & 925581474224825543355918293691083787700 & 26 & 497763403637168199467209054344086921612250 \\
27 & 2389558319916599438092598184185848855064805 & 28 & 1490015653548648524331298995427104771134876780 \\
29 & 6022658118087068522072296623269607788900382000 & 30 & 4298207770675687389258467568039717251504125628050 \\
\bottomrule
\end{tabular}
\caption{Real Gopakumar--Vafa invariants of the quintic, $g=4,5$ ($\chi=3,4$), $d\le 30$.}%
\label{tab:quintic-gv-45}
\end{table}

\begin{table}[ht]
\centering
\fontsize{6}{7.5}\selectfont
\setlength{\tabcolsep}{4pt}
\begin{tabular}{rrr}
\toprule
$d$ & $n^{\mathbb R}_{6,d}$ \ ($n^{\mathbb R}_{6,5}=2$) & $n^{\mathbb R}_{6,d}$ \ ($n^{\mathbb R}_{6,5}=-2$) \\
\midrule
1 & 0 & 0 \\
3 & 0 & 0 \\
5 & 2 & $-2$ \\
7 & $-2110$ & 3190 \\
9 & $-573978050$ & $-576516050$ \\
11 & $-22336444840775$ & $-22335409750775$ \\
13 & $-276507245353062185$ & $-276507291791162185$ \\
15 & $-2096870088848795432851$ & $-2096869808580739558387$ \\
17 & $-12025848519818531849937470$ & $-12025848202834576238514170$ \\
19 & $-57486679661068463609678134265$ & $-57486679184769394788197787765$ \\
21 & $-241632728499728918822738814145560$ & $-241632727782124585509524743992360$ \\
23 & $-922798742486602426523809095545774840$ & $-922798741370524997043788278439549840$ \\
25 & $-3271854171907618934691410665222504721285$ & $-3271854170132026352081651678926082649885$ \\
27 & $-10932794318777291405045959889912217185343315$ & $-10932794315897962942747715252684674759638915$ \\
29 & $-34804504839287002240085192362565712527711956350$ & $-34804504834542253847053498544469097685231602350$ \\
\bottomrule
\end{tabular}
\caption{Real Gopakumar--Vafa invariants of the quintic, $g=6$ ($\chi=5$), $d\le 30$, for the two choices of $n^{\mathbb R}_{6,5}$.}%
\label{tab:quintic-gv-6}
\end{table}

\begin{table}[ht]
\centering
\fontsize{6}{7.5}\selectfont
\setlength{\tabcolsep}{4pt}
\begin{tabular}{rrr}
\toprule
$d$ & $n^{\mathbb R}_{7,d}$ \ ($n^{\mathbb R}_{6,5}=2$) & $n^{\mathbb R}_{7,d}$ \ ($n^{\mathbb R}_{6,5}=-2$) \\
\midrule
2 & 0 & 0 \\
4 & 0 & 0 \\
6 & 0 & 0 \\
8 & 193035 & 186135 \\
10 & $-25741371600$ & $-25734041760$ \\
12 & $-1674089561261385$ & $-1674094428693585$ \\
14 & $-30443700186413394545$ & $-30443699563165500545$ \\
16 & $-317762466694428868716230$ & $-317762468737303822785230$ \\
18 & $-2389601151152336765328472185$ & $-2389601153711614454223799785$ \\
20 & $-14435484556379124698405672383720$ & $-14435484560954527807382533185520$ \\
22 & $-74510914812998634848602947136102780$ & $-74510914820879840808076491517919980$ \\
24 & $-341567909374387905880752903497377302255$ & $-341567909388211153321482722969079020055$ \\
26 & $-1427088890779043827679112806887669923003530$ & $-1427088890803524178830561894237878549055530$ \\
28 & $-5534277227863197152601275025255546763999699055$ & $-5534277227906935377661031069258596987924313255$ \\
30 & $-20188224866531304679803091937161269797567951347815$ & $-20188224866610042347918402038043516092728004184295$ \\
\bottomrule
\end{tabular}
\caption{Real Gopakumar--Vafa invariants of the quintic, $g=7$ ($\chi=6$), $d\le 30$, for the two choices of $n^{\mathbb R}_{6,5}$.}%
\label{tab:quintic-gv-7}
\end{table}

\begin{table}[ht]
\centering
\fontsize{6}{7.5}\selectfont
\setlength{\tabcolsep}{4pt}
\begin{tabular}{rrr}
\toprule
$d$ & $n^{\mathbb R}_{8,d}$ \ ($n^{\mathbb R}_{6,5}=2$) & $n^{\mathbb R}_{8,d}$ \ ($n^{\mathbb R}_{6,5}=-2$) \\
\midrule
1 & 0 & 0 \\
3 & 0 & 0 \\
5 & 0 & 0 \\
7 & 0 & 0 \\
9 & $-4498215$ & $-3722415$ \\
11 & 413127892630 & 412239166630 \\
13 & 27024796890661690 & 27025277309082090 \\
15 & 511984204393090577600 & 511983777326966592340 \\
17 & 5649589245998794394433705 & 5649588752022721967967605 \\
19 & 45136907631169163000488005770 & 45136906265868768452713847370 \\
21 & 289925814085028058172268323254740 & 289925811041300738191124386771740 \\
23 & 1589607490077113505657634024935134835 & 1589607483411264769780776529858944935 \\
25 & 7726309026585241955513503632136293188225 & 7726309012320903810624676026609061450725 \\
27 & 34153430887154621589919564672211768826755950 & 34153430857074318497562854292961174683362650 \\
29 & 139814327790584713542910271137026975249044571650 & 139814327727825109346606071110379408157435699750 \\
\bottomrule
\end{tabular}
\caption{Real Gopakumar--Vafa invariants of the quintic, $g=8$ ($\chi=7$), $d\le 30$, for the two choices of $n^{\mathbb R}_{6,5}$.}%
\label{tab:quintic-gv-8}
\end{table}

\begin{table}[p]
\centering
\scriptsize
\setlength{\tabcolsep}{2.5pt}
\renewcommand{\arraystretch}{0.95}
\begin{tabular}{rr@{\hspace{1.2em}}rr}
\toprule
$d$ & $n^{\mathbb R}_{0,d}$ & $d$ & $n^{\mathbb R}_{1,d}$ \\
\midrule
1 & 9 & 2 & 0 \\
3 & 90 & 4 & 693 \\
5 & 15759 & 6 & 381912 \\
7 & 3297987 & 8 & 167597505 \\
9 & 841201389 & 10 & 70912518192 \\
11 & 241496789706 & 12 & 29843206833573 \\
13 & 75029122895328 & 14 & 12591774914348088 \\
15 & 24662585381363001 & 16 & 5337998490724886529 \\
17 & 8456520108661667946 & 18 & 2274458742050165147400 \\
19 & 2996347397798754061326 & 20 & 973844673198672245668416 \\
21 & 1089868924565421236420841 & 22 & 418831614404968713981811560 \\
23 & 405000793360436396494255011 & 24 & 180858895037533129448973514257 \\
25 & 153206455414898076025297635297 & 26 & 78381436853507144470910401790592 \\
27 & 58835468751409198689645123627273 & 28 & 34079961019715434110261068388115713 \\
29 & 22887668860106412716760233910613338 & 30 & 14861262406840812199145667595997399640 \\
\bottomrule
\end{tabular}
\caption{Real Gopakumar--Vafa invariants of the bicubic, $g=0,1$ ($\chi=-1,0$), $d\le 30$.}%
\label{tab:bicubic-gv-01}
\end{table}

\begin{table}[p]
\centering
\scriptsize
\setlength{\tabcolsep}{2.5pt}
\renewcommand{\arraystretch}{0.95}
\begin{tabular}{rr@{\hspace{1.2em}}rr}
\toprule
$d$ & $n^{\mathbb R}_{2,d}$ & $d$ & $n^{\mathbb R}_{3,d}$ \\
\midrule
1 & 0 & 2 & 0 \\
3 & 0 & 4 & 0 \\
5 & $-693$ & 6 & $-90$ \\
7 & $-568557$ & 8 & $-3465342$ \\
9 & $-328426623$ & 10 & $-6223431780$ \\
11 & $-175272593346$ & 12 & $-6586918712244$ \\
13 & $-89789350783752$ & 14 & $-5480984828361042$ \\
15 & $-44869010445376389$ & 16 & $-3977616521462805252$ \\
17 & $-22052256784116464466$ & 18 & $-2644109379105225418062$ \\
19 & $-10711077115968546752799$ & 20 & $-1654013788366163106634920$ \\
21 & $-5157161549849152593898263$ & 22 & $-989699117575942119575532006$ \\
23 & $-2466460720670287348412238762$ & 24 & $-572516522772993260360980775298$ \\
25 & $-1173399265840401421711616225907$ & 26 & $-322519359074944001846113211329872$ \\
27 & $-555874087393454301039248647944426$ & 28 & $-177853573161972651052657749309118668$ \\
29 & $-262422992493867519522401571053318313$ & 30 & $-96376631267356215547046740131881251530$ \\
\bottomrule
\end{tabular}
\caption{Real Gopakumar--Vafa invariants of the bicubic, $g=2,3$ ($\chi=1,2$), $d\le 30$.}%
\label{tab:bicubic-gv-23}
\end{table}

\begin{table}[p]
\centering
\scriptsize
\setlength{\tabcolsep}{2.5pt}
\renewcommand{\arraystretch}{0.95}
\begin{tabular}{rr@{\hspace{1.2em}}rr}
\toprule
$d$ & $n^{\mathbb R}_{4,d}$ & $d$ & $n^{\mathbb R}_{5,d}$ \\
\midrule
1 & 0 & 2 & 0 \\
3 & 0 & 4 & 0 \\
5 & 0 & 6 & 0 \\
7 & 486 & 8 & 531 \\
9 & 9129930 & 10 & 112930776 \\
11 & 18944223273 & 12 & 452773779129 \\
13 & 22937859003321 & 14 & 899317173765132 \\
15 & 21585466396474539 & 16 & 1248803221560269274 \\
17 & 17522646108367547979 & 18 & 1392964804139982886476 \\
19 & 12903123110276661075792 & 20 & 1340073665151925714576929 \\
21 & 8865557848861755843375765 & 22 & 1158874281900367107309230796 \\
23 & 5783991815107616596266826827 & 24 & 924809309038093118575776591267 \\
25 & 3624915555585279079454413062639 & 26 & 693176618824718777907165010430316 \\
27 & 2200054635694421457270839610496578 & 28 & 494102685297484943024601164683997358 \\
29 & 1300730730008018694499260738150236280 & 30 & 338005733489707824720132058393801959384 \\
\bottomrule
\end{tabular}
\caption{Real Gopakumar--Vafa invariants of the bicubic, $g=4,5$ ($\chi=3,4$), $d\le 30$.}%
\label{tab:bicubic-gv-45}
\end{table}

\begin{table}[p]
\centering
\scriptsize
\setlength{\tabcolsep}{2.5pt}
\renewcommand{\arraystretch}{0.95}
\begin{tabular}{rr@{\hspace{1.2em}}rr}
\toprule
$d$ & $n^{\mathbb R}_{6,d}$ & $d$ & $n^{\mathbb R}_{7,d}$ \\
\midrule
1 & 0 & 2 & 0 \\
3 & 0 & 4 & 0 \\
5 & 0 & 6 & 0 \\
7 & 0 & 8 & $-18$ \\
9 & $-1404$ & 10 & 8991 \\
11 & $-452104281$ & 12 & $-7461323259$ \\
13 & $-1950969312603$ & 14 & $-56551941465030$ \\
15 & $-4212540620573814$ & 16 & $-183930175211466114$ \\
17 & $-6360356680057713753$ & 18 & $-385632982358776268919$ \\
19 & $-7691298186253072608504$ & 20 & $-613698753034561852037169$ \\
21 & $-7989689308213682600722830$ & 22 & $-807099207336816509613430350$ \\
23 & $-7429512345382259574495058467$ & 24 & $-922657487757200612323500139710$ \\
25 & $-6349436481346659179612822318715$ & 26 & $-947291950869136381601887705666542$ \\
27 & $-5077298829915835644916385130369672$ & 28 & $-893269894174768664668363918437551664$ \\
29 & $-3847566549404004040534425604421939361$ & 30 & $-786147898433782547134066672770555339315$ \\
\bottomrule
\end{tabular}
\caption{Real Gopakumar--Vafa invariants of the bicubic, $g=6,7$ ($\chi=5,6$), $d\le 30$.}%
\label{tab:bicubic-gv-67}
\end{table}

\begin{table}[p]
\centering
\scriptsize
\setlength{\tabcolsep}{2.5pt}
\renewcommand{\arraystretch}{0.95}
\begin{tabular}{rr@{\hspace{1.2em}}rr}
\toprule
$d$ & $n^{\mathbb R}_{8,d}$ & $d$ & $n^{\mathbb R}_{9,d}$ \\
\midrule
1 & 0 & 2 & 0 \\
3 & 0 & 4 & 0 \\
5 & 0 & 6 & 0 \\
7 & 0 & 8 & 0 \\
9 & 117 & 10 & 18 \\
11 & $-94167$ & 12 & 2098863 \\
13 & 41134610916 & 14 & 1099444123116 \\
15 & 322656117414678 & 16 & 12130825858886664 \\
17 & 1103410184875751106 & 18 & 56522911777127800038 \\
19 & 2449277640353245268835 & 20 & 163946278179334272081096 \\
21 & 4135021184881633176440586 & 22 & 349433462040407154072938745 \\
23 & 5767506742075560994283749899 & 24 & 598656687256980638827512658218 \\
25 & 6983099250993390578423272194825 & 26 & 870866275498983674669150768576907 \\
27 & 7579321616055330992335120970676930 & 28 & 1115351534268377802592620625264437612 \\
29 & 7540057172588684946098760785904417369 & 30 & 1289828720158755688399776034449705347151 \\
\bottomrule
\end{tabular}
\caption{Real Gopakumar--Vafa invariants of the bicubic, $g=8,9$ ($\chi=7,8$), $d\le 30$.}%
\label{tab:bicubic-gv-89}
\end{table}

\begin{table}[p]
\centering
\scriptsize
\setlength{\tabcolsep}{2.5pt}
\renewcommand{\arraystretch}{0.95}
\begin{tabular}{rrr}
\toprule
$d$ & $n^{\mathbb R}_{10,d}$, $n^{\mathbb R}_{10,9}=3$ & $n^{\mathbb R}_{10,d}$, $n^{\mathbb R}_{10,9}=-3$ \\
\midrule
1 & 0 & 0 \\
3 & 0 & 0 \\
5 & 0 & 0 \\
7 & 0 & 0 \\
9 & 3 & $-3$ \\
11 & $-810$ & 2106 \\
13 & $-937080$ & $-1501326$ \\
15 & $-7882582875564$ & $-7882513387608$ \\
17 & $-87400261892751750$ & $-87400266979774986$ \\
19 & $-418294824539306480082$ & $-418294823994564417270$ \\
21 & $-1260079376139477408155553$ & $-1260079376097579634722051$ \\
23 & $-2803789890232313576869945998$ & $-2803789890210185276664046014$ \\
25 & $-5025551327410498077310670325861$ & $-5025551327402899352587731447921$ \\
27 & $-7652927560219576398637084649618985$ & $-7652927560216772703097620220917243$ \\
29 & $-10257055869055199434701843827552583438$ & $-10257055869054150146507824533144186144$ \\
\bottomrule
\end{tabular}
\caption{Real Gopakumar--Vafa invariants of the bicubic, $g=10$ ($\chi=9$), $d\le 30$, for the two values of $n^{\mathbb R}_{10,9}$.}%
\label{tab:bicubic-gv-10}
\end{table}

\begin{table}[p]
\centering
\scriptsize
\setlength{\tabcolsep}{2.5pt}
\renewcommand{\arraystretch}{0.95}
\begin{tabular}{rrr}
\toprule
$d$ & $n^{\mathbb R}_{11,d}$, $n^{\mathbb R}_{10,9}=3$ & $n^{\mathbb R}_{11,d}$, $n^{\mathbb R}_{10,9}=-3$ \\
\midrule
2 & 0 & 0 \\
4 & 0 & 0 \\
6 & 0 & 0 \\
8 & 0 & 0 \\
10 & 0 & 0 \\
12 & 1314 & $-90$ \\
14 & $-2232033111$ & $-2231444079$ \\
16 & $-299775099190833$ & $-299775216863097$ \\
18 & $-4155217198362443538$ & $-4155217185419663922$ \\
20 & $-25233487661019885895845$ & $-25233487662653704652013$ \\
22 & $-94634641356961048270265979$ & $-94634641357044710128841043$ \\
24 & $-256676121714813802901671094676$ & $-256676121714886684744440268260$ \\
26 & $-550297398498756656087726195192073$ & $-550297398498784445638641529853577$ \\
28 & $-986172662784564371133615953678957631$ & $-986172662784575976594639906951685683$ \\
30 & $-1534006017802059698701295033835296160702$ & $-1534006017802064530682013147905237704470$ \\
\bottomrule
\end{tabular}
\caption{Real Gopakumar--Vafa invariants of the bicubic, $g=11$ ($\chi=10$), $d\le 30$, for the two values of $n^{\mathbb R}_{10,9}$.}%
\label{tab:bicubic-gv-11}
\end{table}

\begin{table}[p]
\centering
\scriptsize
\setlength{\tabcolsep}{2.5pt}
\renewcommand{\arraystretch}{0.95}
\begin{tabular}{rrr}
\toprule
$d$ & $n^{\mathbb R}_{12,d}$, $n^{\mathbb R}_{10,9}=3$ & $n^{\mathbb R}_{12,d}$, $n^{\mathbb R}_{10,9}=-3$ \\
\midrule
1 & 0 & 0 \\
3 & 0 & 0 \\
5 & 0 & 0 \\
7 & 0 & 0 \\
9 & 0 & 0 \\
11 & 0 & 0 \\
13 & $-29160$ & 79218 \\
15 & 20989564065 & 20945596779 \\
17 & 2676588562043271 & 2676597160615767 \\
19 & 36482537488337918127 & 36482536334345350869 \\
21 & 223805406812095596838455 & 223805406851964010295649 \\
23 & 859480127215673942541360375 & 859480127173727625841196193 \\
25 & 2404126862475882848884492401933 & 2404126862457055846585283889153 \\
27 & 5335670295274823737557358514277717 & 5335670295264726348696131892133527 \\
29 & 9917036170549689198004286519041081632 & 9917036170544572941046180419802870434 \\
\bottomrule
\end{tabular}
\caption{Real Gopakumar--Vafa invariants of the bicubic, $g=12$ ($\chi=11$), $d\le 30$, for the two values of $n^{\mathbb R}_{10,9}$.}%
\label{tab:bicubic-gv-12}
\end{table}

\clearpage

\bibliography{literature.bib}{}
\bibliographystyle{alpha}
\end{document}